\documentclass[11pt,reqno]{amsart}
\usepackage{amsmath}
\usepackage{amssymb}
\usepackage{latexsym}
\usepackage{mathrsfs}
\def\squarebox#1{\hbox to #1{\hfill\vbox to #1{\vfill}}}

\DeclareMathOperator{\re}{Re}
\DeclareMathOperator{\im}{Im}

\DeclareMathOperator{\op}{Op}

\DeclareMathOperator{\symb}{symb}
\newcommand{\1}{{\bold 1}}

\theoremstyle{plain}

\newtheorem{thm}{Theorem}
\newtheorem{cor}{Corollary}
\newtheorem{lem}{Lemma}
\newtheorem{rem}{Remark}
\newtheorem{prop}{Proposition}
\newtheorem{defn}{Definition}

\renewcommand{\Re}{\mathop{\rm Re}\nolimits}
\renewcommand{\Im}{\mathop{\rm Im}\nolimits}

\def\pu{\partial_t u}
\def\ppu{\partial_t^2 u}
\def\la{\langle}
\def\ra{\rangle}
\def\el{e^{-2\lambda t}}
\def\ell{e^{-\lambda t}}
\def\12{\frac{1}{2}}

\def\ii{{\bf i}}
\def\phi{\varphi}
\def\epsilon{\varepsilon}
\def\kappa{\varkappa}
\def\ep{\epsilon}

\def\b0{{\bf b}_2^{\ep}}

\numberwithin{equation}{section}

\begin{document}

\def\b0{{\bf b}_2^{\ep}}
\def\R {{\mathbb{R}}}
\def\N {{\mathbb{N}}}
\def\C {{\mathbb{C}}}
\def\Z {{\mathbb{Z}}}
\def\T{{\mathbb T}}
\def\Q{{\mathbb Q}}
\def\SP{{\mathbb S}}
\def\d{{\partial}}
\def\mc{{\mathcal H}}
\def\1b{{\mathbb I}}
\def\pv{\partial_x V}
\def\pu{\partial_t u}
\def\ppu{\partial_t^2 u}
\def\es{e^{-\lambda s }}
\def\a2e{a_2^{\epsilon}(t, x, D_x)}
\def\3a{a_3^{\epsilon}(t, x, D_x)}
\def\oa{a_1^{\epsilon}(t, x, D_x)}
\def\pt{\partial_t}
\def\pa{\partial}
\def\aee{a_2^{\ep}}
\def\N{\frac{2N}{3}}
\def\aez{\alpha_0^{\ep}}
%


%
%
 \vbadness=10000
 \hbadness=10000
%
%

%
%
%

%

\def\phi{\varphi}
\def\epsilon{\varepsilon}
\def\kappa{\varkappa}
\def\bv {b_1(t,\xi)u}
\def\ia{t^2 a_3(t, \xi) \im (u \bar{u}'')}
\def\buu{\bar{u}''}
\def\bu{\bar{u}'}
\def\eT{e^{-\lambda T}}
\def\ii{{\bf i}}
\def\hh{\hat{x}}
\def\hhx{\hat{\xi}}
\def\12{\frac{1}{2}}
\def\Rc{{\mathcal R}}
\def\tE{\tilde{E}}
\def\ep{\epsilon}
\def\ae{\alpha^{\ep}}
\def\la{\langle}
\def\ra{\rangle}
\def\a2{a_2^{\ep}}
\def\fN{f^{-2 N}}
\def\fNt{f^{-N - 1/2}}
\def\fr{\frac}
%
\def\Qed{\qed\par\medskip\noindent}
%
\title[Cauchy problem for effectively hyperbolic operators ]{Cauchy problem for
  effectively hyperbolic operators with triple characteristics of variable
  multiplicity} 
\date{}
\author[E. Bernardi, A. Bove] {Enrico Bernardi, Antonio Bove}
\author[V. Petkov]{Vesselin Petkov}

\address{Dipartimento di Scienze Statistiche, Universit\`a di Bologna,
Viale Filopanti 5, 40126 Bologna, Italia}

\email{enrico.bernardi@unibo.it}
\address{Dipartimento di Matematica,  Universit\`a di Bologna,\:
Piazza di Porta S. Donato 5, 
40127 Bologna, Italia}
\email{bove@bo.infn.it}

\address{Universit\'e Bordeaux I, Institut de Math\'ematiques de Bordeaux, 351, 
Cours de la Lib\'eration, 33405  Talence, France}
\email{petkov@math.u-bordeaux1.fr}
\thanks{VP was partially supported by ANR project Nosevol BS01019 01}
\begin{abstract}
We study a class of third order hyperbolic operators $P$ in $G = \{(t, x):0 \leq t \leq T,\: x\in U \Subset \R^{n}\}$ with
triple characteristics at $\rho = (0, x_0, \xi),\: \xi \in \R^n \setminus \{0\}$. We consider the case when the
fundamental matrix of the principal symbol of $P$ at $\rho$ has a couple of
non-vanishing real eigenvalues. Such operators are called {\it effectively hyperbolic}. V. Ivrii introduced the conjecture that every effectively hyperbolic operator is {\it strongly hyperbolic}, that is the Cauchy problem for $P + Q$ is locally well posed for any lower order terms $Q$. This conjecture has been solved for operators having at most double characteristics and for operators with triple characteristics in the case when the principal symbol admits a factorization. A strongly hyperbolic operator in $G$ could have triple characteristics in $G$ only for $t = 0$ or for $t = T$.  We prove that the operators in our class are strongly hyperbolic if $T$ is small enough. Our proof is based on energy estimates with a loss of regularity. 
\end{abstract}
\subjclass[2010]{Primary 35L30, Secondary 35L25}
\keywords{Cauchy Problem, Effectively Hyperbolic Operators, Triple
  Characteristics, Energy Estimates}
\vspace{0.5cm}
\maketitle

\tableofcontents

\section{Introduction}
%
%
\setcounter{equation}{0}
\setcounter{thm}{0}
\setcounter{prop}{0}  
\setcounter{lem}{0}
\setcounter{cor}{0} 
\setcounter{defn}{0}
\setcounter{rem}{0}
%

\subsection{Notations and main result}

Consider a differential operator
\begin{equation}
P(t, x, D_t, D_x) = \sum _{\alpha + |\beta| \leq m} c_{\alpha, \beta}
(t, x) D_t^{\alpha} D_x^{\beta},\: D_t = -\ii\partial_t, D_{x_j} =
-\ii\partial_{x_j}
\end{equation}
of order $m$ with $C^{\infty}$ coefficients $c_{\alpha, \beta}(t,x),\:
t \in \R,\:x \in \R^n.$ Denote by 
$$
p_m(t, x, \tau, \xi) = \sum_{\alpha + |\beta| = m} c_{\alpha, \beta}
(t, x) \tau^{\alpha} \xi^{\beta} 
$$
the principal symbol of $P$. We assume that $c_{m, 0}(t, x) \neq 0$ for all $(t, x)$.
Let $\Omega \subset \R^{n + 1}$ be an open set and let 
$\Omega_{\eta}^{-} = \Omega \cap \{t \leq \eta\},\Omega_{\eta}^{+} =
\Omega \cap \{t \geq \eta\}, \:G =  \Omega \cap \{0 \leq t \leq T\}.$

Set $P_m(t, x, D_t, D_x) = p_m(t, x, D_t, D_x).$

\begin{defn} We say that the Cauchy problem
\begin{equation} \label{eq:1.2}
Pu = f\: {\rm in}\: \Omega \cap \{t < T\},\: {\rm supp}\: u \subset \overline{G}
\end{equation}
is well posed in $G$ if\\
$(i)$ $($existence$)$ for every $f \in C_0^{\infty}(\Omega), \: {\rm supp}\: f \subset \overline{\Omega_0^{+}}$ there
exists a solution $u \in {\mathcal D}'(\Omega)$ satisfying $(\ref{eq:1.2})$.\\
$(ii)$ $($uniqueness$)$ if $u \in {\mathcal D}'(\Omega)$ satisfies $(\ref{eq:1.2})$, then for every $s, 0 < s \leq T,$ if $Pu = 0$ in $\Omega_s^{-}$, then $u = 0$ in $\Omega_s^{-}$.
\end{defn}
A necessary condition for the well posedeness of the Cauchy problem (WPC) is the hyperbolicity of the operator $P$ in $ G$ (see \cite{IP} and the references cited there).  This means that for every $(t_0, x_0, \xi) \in G \times \R^n \setminus \{0\}$ the equation 
\begin{equation} \label{eq:1.1}
p_m(t_0, x_0, \tau, \xi) = 0
\end{equation}
 with respect to $\tau $ has only real roots  $\tau = \lambda_j(t_0,
 x_0, \xi)$.

\begin{defn} We say that the operator $P$ with principal symbol $p_m$
  is strongly hyperbolic in $G$ if for every point $z_0 = (t_0, x_0)
  \in G$ there exists a neighbourhood $U$ of $z_0$, $T(U) > 0$  and $T_0 \geq 0$
  ($T_0 < T$ if $t_0 = T$ and $T_0= 0$ if $t_0 = 0$) such that the
  Cauchy problem $(\ref{eq:1.2})$ for the operator $L = P_m(t, x, D_t,
  D_x) + Q_{m-1}(t, x, D_t, D_x)$ is well posed in $U^{+}_s$ for every
  $T_0 \leq s < T(U)$ and for any operator $Q_{m-1}(t, x, D_t, D_x)$
  of order less or equal to $m-1$. 
\end{defn}
When $P$ is strictly hyperbolic, that is when 
the equation (\ref{eq:1.1}) has simple roots $\lambda_j(t, x, \xi)$
with respect to the variable $ 
\tau $ for  all $(t, x, \xi)\in G \times \R^N \setminus \{0\}$, it is a classical
result that $P$ is strongly hyperbolic. If the equation (\ref{eq:1.1})
has real roots 
with constant multiplicity for $(t, x, \xi) \in G \times \R^n
\setminus \{0\}$, the operator $P$ is strongly hyperbolic {\bf if and
  only if} it is strictly hyperbolic.  Thus in the case of roots with
constant multiplicity---greater than 1---we must impose conditions on
the lower order terms $Q_{m-1}$, called Levi conditions, in order that
the Cauchy problem be well posed. The
analysis of the Cauchy problem for such operators is complete and we
know the necessary \cite{FS} and sufficient \cite{Ch} conditions for
(WPC).

Passing to the case when the roots of (\ref{eq:1.1}) have variable
multiplicity, notice that the roots $\lambda_j(t, x, \xi)$ in
general are not smooth but only continuous. The case of operators with
constant coefficients is also completely examined and $P$ is strongly
hyperbolic {\bf if and only if} $P$ is strictly hyperbolic. The
necessary and sufficient condition of G\aa rding for (WPC) says that
there exists a constant $c > 0$ such that for the full symbol $p(\tau, \xi)$ of
$P$ we have
$$p(\tau, \xi) \neq 0,\:{\rm for}\: |{\rm Im}\: \tau| > c,\: \forall
\xi \in \R^n \setminus \{0\}.$$

In the following, for the sake of simplicity, we switch to a different notation and  denote $t = x_0$,
$x = (x_0, x_1,...,x_n) \in \R^{n+1}.$ The dual variables are denoted by
$\xi = (\xi_0, \xi_1,...,\xi_n) = (\xi_0, \xi').$

Given a symbol $p(x, \xi)$,  let 
$$\Sigma(p) = \{z \in
 T^*{G} \setminus \{0\}:\: p(z) = 0\},\: \Sigma_1(p) = \{ z \in
 T^*(G) \setminus \{0\}: p(z) = 0,\: dp(z) = 0\}.$$
In the case $\Sigma_1(p_m) = \emptyset,$ the operator is of principal
type and a hyperbolic operator $P$ in $G$ is strongly hyperbolic (see \cite{I} and
Section 23.4 in \cite{h3}). 

Turning to the case $\Sigma_1(p_m) \neq \emptyset,$ notice that if we have a critical
 point $(\hh, \hhx) \in \Sigma_1(p)$, then the Hamiltonian system 
$$\frac{dx}{ds} = \partial_{\xi} p,  \qquad  \frac{d \xi}{ds} = -\partial_{x} p$$
has a stationary point and it is natural to consider the differential
of the Hamilton vector field. Thus we are led to define the
fundamental matrix
$$
F_{p} (\hat{x}, \hat{\xi}) = \left(\begin{matrix} p_{\xi, x}(\hat{x}, \hat{\xi}) & &  p_{\xi, \xi}(\hat{x}, \hat{\xi})\\
-p_{x, x}(\hat{x}, \hat{\xi}) & &  - p_{x,
  \xi}(\hat{x},\hat{\xi})\end{matrix}\right).
$$
We recall two important properties of $F_p$ (see \cite{IP}, \cite{H1}):
\begin{itemize}
\item[1. ]{} 
For every  point $z \in\Sigma_1(p)$ the Hessian $Q_p(X, Y),\: X, Y \in T_z(T^*(G))$ at $z$ of $\frac{p}{2}$ is well defined. Then
$ Q_p(X, Y) = \sigma(X, F_p(z) Y),$ $\sigma$ being the symplectic form on $T^*(G).$
Thus after a canonical transformation the fundamental matrix is
transformed into a similar one and its eigenvalues are invariant under
canonical transformations. H\"ormander \cite{H1} called $F_p(z)$ the
Hamilton map of $Q_p$.
\item[2. ]{} 
If $P$ is hyperbolic in $G$ and $(\hat{x}, \hat{\xi})$ is a
critical point of $p_m(x, \xi)$, then $F_{p_m}(\hat{x}, \hat{\xi})$
has at most two non-vanishing real simple eigenvalues $\mu$ and $-\mu$
and all other eigenvalues $\mu_j$ are purely imaginary, that is
${\rm Re}\: \mu_j = 0.$
\end{itemize}
 
The existence of non-vanishing real eigenvalues of $F_{p_m}(\hh,
\hhx)$ is a {\it necessary condition} for strong hyperbolicity. More
precisely, let $p_{m-1}(x, \xi) = \sum_{|\alpha| = m-1} c_{\alpha}(x)
\xi^{\alpha}$ and let 
$$p_{m-1}'(x, \xi) = p_{m-1}(x, \xi) + \frac{i}{2} \sum_{j=0}^n
\frac{\partial^2 p_{m}}{\partial x_j \partial\xi_j} (x, \xi)$$ 
 be the subprincipal symbol of $P$ which is invariantly defined for
 $(x, \xi) \in \Sigma_1(p_m).$ Then we have the following
\begin{thm} [Theorem 3 and Corollary 3 in \cite{IP}] 
\label{thm:ip}
If $P$ is strongly hyperbolic in $G$, then at
  every  point $(\hh, \hhx) \in \Sigma_1(p_m)$ the fundamental matrix
  $F_{p_m}(\hat{x}, \hat{\xi})$ has two non-vanishing real
  eigenvalues. Moreover, for $(x, \xi') \in \overset{\circ}G \times
  (\R^n \setminus \{0\})$ the multiplicities of the roots of $(1)$ are
  not greater than two, and for $(x, \xi') \in \{x_0 = 0\} \times \R^n
  \setminus \{0\}$ or for $(x, \xi') \in \{x_0 = T\} \times \R^n
  \setminus \{0\}$ these multiplicities are not greater than three.
If $F_{p_m}(\hh, \hhx)$ has only purely imaginary eigenvalues, the
condition $\im\: p_{m-1}'(\hh, \hhx) = 0$  is necessary for
(WPC).
\end{thm}
If $F_{p_m}(\hh, \hhx)$ has only purely imaginary eigenvalues, for
(WPC) we have a second necessary condition 
$$|{\rm Re}\: p_{m-1}'(\hh, \hhx)| \leq \frac{1}{4} \sum_{j=0}^{2n + 2}|\mu_j|,$$
$\mu_j$ being the eigenvalues of $F_{p_m}(\hh, \hhx)$ repeated
according to their multiplicities. This condition has been proved in
\cite{IP} in some special cases concerning the structure of
$F_{p_m}(\hh, \hhx)$ and without any restriction by H\"ormander
\cite{H1}. 
\begin{defn} A hyperbolic operator with principal symbol $p_m(x, \xi)$
  is called effectively hyperbolic if at every point $(\hh, \hhx)
  \in \Sigma_1(p_m)$, the fundamental matrix $F_{p_m}(\hh, \hhx)$ has two
  non-vanishing real eigenvalues. 
\end{defn} 
V. Ivrii introduced the following

\medskip

{\bf Conjecture.}
\label{conj:ivrii}
{\it A hyperbolic operator is strongly hyperbolic if
and only if it is effectively hyperbolic}.

\medskip

For operators with at most double characteristics  the sufficient part
of the above conjecture has been established for some special class of effectively operators by Oleinik \cite{O}, H\"ormander \cite{H1}, Ivrii \cite{I}, Melrose \cite{M} and in the general case by N. Iwasaki \cite{Iw1}, \cite{Iw2} and T. Nishitani
\cite{N1}, \cite{N2}, \cite{N3}. In particular, in the works of Nishitani many 
 properties of effectively hyperbolic operators with double
 characteristics have been established.  An important phenomenon for
 this class of operators is that we have always a loss of regularity
 which depends on the ratio of the subprincipal symbol and the
 non-vanishing real eigenvalue of $F_{p_m}$ at double characteristic
 points (see Theorem 3 in \cite{IP} for a necessary condition, and \cite{I}, \cite{N2} for sufficient conditions).

However according to Theorem \ref{thm:ip}, there are also classes
of effectively hyperbolic operators with characteristics of
multiplicity 3 which should be strongly hyperbolic. 
The analysis of operators with double characteristics is fairly
complete and over the last 30 years a lot of papers
treating both the effectively hyperbolic and the non-effectively
hyperbolic operators has appeared. The case when we have triple characteristics is more complicated
for several reasons. So far the  only result concerning effectively hyperbolic operators with triple characteristics appears to be that of Ivrii \cite{I} for operators such that  in a conic neighborhood of every point $(0, x_0, \xi_0)$ with triple characteristics $p_3$ admits a factorization
\begin{equation} \label{eq:1.4}
p_3(t, x, \tau, \xi) = ((\tau- \beta(x, \xi))^2- D (t, x, \xi))(\tau- \gamma(t, x, \xi))
\end{equation} 
as a product of two principal type symbols with smooth real-valued
symbols $\beta, \gamma, D$ and $D \geq 0$  for $t \geq 0.$ For such a class of operators
Ivrii established in \cite{I} a  strong hyperbolicity result. It is clear that,
for a factorization of the form (\ref{eq:1.4}) to exist, it is
necessary that the equation $p_3 = 0$ has a
$C^{\infty}$ real root $\tau = \gamma(t, x, \xi)$ defined in a small
conic neighborhood of the point $(0, x_0, \xi_0)$. 
This is possible only if some terms of $p_3$ have a very
special behavior. For example, consider the symbol 
\begin{equation} \label{eq:1.5}
p_3 = \tau^3 - (t a_2(t, x, \xi) + \alpha(x, \xi))\tau + t^2 b_3(t, x, \xi)
\end{equation}
with $a_2(t, x, \xi) \geq c |\xi|^2, \: c > 0, \:\alpha(x, \xi) \geq 0$ and
$\alpha(0, \xi_0) = 0, 4(ta_2 + \alpha)^3 \geq 27 t^4 b_2^3$ for $t
\geq 0$, where $b_3(t, x, \xi)$ is a symbol of third order.
Then if $b_3(0, 0, \xi_0) \neq 0$, the factorization (\ref{eq:1.4}) for fixed $\xi = \xi_0$ is
possible only if $\alpha(x, \xi_0) \equiv 0$ for all $x$ in a neighborhood of $0.$
We prove this result in the Appendix. 

The purpose of the present paper is to show that for a class of third
order weakly hyperbolic operators  whose principal symbol $p_3$, after a
reduction, has the form (\ref{eq:1.5}) we have
strong hyperbolicity if the Hamilton map $F_{p_3}$ has two real non-vanishing 
eigenvalues in its spectrum on the triple characteristic points. The last condition is
satisfied if the symbol $a_2(t, x, \xi)$ is elliptic. According to
Theorem 1.1, a strongly  hyperbolic operator may have triple characteristics only for $t = 0$
or $t = T$ and in this paper we deal with the case when this may  happen
for some points on $t = 0.$

More precisely, we study operators having the form 
\begin{eqnarray} 
\label{1.5}
P & = & D_t^3 +   q_1(t, x, D_x)D_t^2 +q_2(t, x, D_x)D_t + q_3(t, x, D_x) \nonumber \\
&  &
+  r_2(t,x, D_x) + r_1(t, x, D_x) D_t + r_0(t, x)D_t^2 + m_1(t, x, D_x) + m_0(t, x)D_t + c_0(t, x).
\end{eqnarray}
Here $q_j(t, x, D_x),\: j = 1,2,3,$ are differential operators with
$C^{\infty}$ coefficients and real-valued symbols $q_j(t, x, \xi)$
which are homogeneous polynomials of order $j$ in $\xi$, $r_j(t, x,
D_x),\: j = 1,2,$ are differential operators with $C^{\infty}$
coefficients and symbols $r_j(t, x, \xi)$ homogeneous of order $j$
with respect to $\xi$, $r_0(t, x), m_0(t, x), c_0(t, x)$ are $C^{\infty}$ functions and $m_1(t,
x, D_x)$ is a first order differential operator with $C^{\infty}$
coefficients. Let $p_3(t, x, \tau, \xi)$ be the principal symbol of
$P$ and let $G  = \{(t, x):\: 0 \leq t \leq T,\: x \in  U\},$ where $U
\Subset \R^n$ is an open set in $\R^n.$  
Consider the symbols 
$$\Delta_1  = 27 q_3 - 9q_1 q_2 + 2 q_1^3,\: \Delta_0 = q_1^2 - 3 q_2, \: \Delta =  -\fr{1}{27} (\Delta_1^2 - 4 \Delta_0^3).$$
The symbol $\Delta$ is the {\it discriminant} of the equation $p_3 = 0$ with respect to $\tau$ and we have
three real roots for $t \geq 0$ if and only if $\Delta \geq 0.$ The
symbol $\Delta_0$ is the {\it discriminant} of the equation
$\pa_{\tau} p_3 = 3 \tau^2 + 2 q_1 \tau + q_2 = 0$ with respect to $\tau$ and if we have a
triple root at $\rho = (0, x_0, \xi)$, we get $\Delta_0(\rho) = 0.$
Thus if the equation $p_3 = \tau^3 + q_1 \tau^2 + q_2 \tau + q_3 = 0$
has a triple real root at $\rho = (0, x_0, \xi)$, we must 
have $\Delta_0(\rho) = 0,\: \Delta (\rho) = 0.$ This implies
$\Delta_1(\rho) = 0$ and the triple root is $\tau = -
\fr{q_1(\rho)}{3}.$  

Since the  polynomial $p_3$ is hyperbolic with respect to $\tau$ for
$t \geq 0$, we deduce that at a point $\rho$ with triple
characteristics we have 
$$(d_{t, x} p_3)\Bigl(0, x_0, -\fr{q_1(\rho)}{3}, \xi\Bigr) = 0$$
 (see Lemma 8.1 in \cite{IP}). The last condition can be written as follows
$$(d_{t,x} q_1) (\rho) \fr{q_1^2(\rho)}{9} - (d_{t,x} q_2)(\rho) \fr{q_1(\rho)}{3} + (d_{t,x}q_3)(\rho) = 0.$$
Taking the differential of ${\Delta_1}$ at $\rho$, and using $\Delta_0(\rho) = 0$, we deduce easily that
$(d_{t, x}\Delta_1)(\rho) = 0.$

In this paper we make the following assumptions: 

\bigskip

$(H_0)$ The roots of the equation $p_3(t, x, \tau, \xi) = 0$ with respect to $\tau$ are real for all $(t, x) \in \bar{G}, \xi \in \R^n.$\\

$(H_1)$ If the equation $p_3(0, x, \tau, \xi) = 0$ with respect to $\tau$ has a triple root $\tau= \lambda(0, x_0, \xi_0)$ for $t = 0,\: x_0 \in U,\: \xi_0 \in \R^n \setminus \{0\}$, then we have a triple root $\tau = \lambda(0, x_0, \xi)$ for $(0, x_0, \xi),\forall \xi \in \R^n \setminus \{0\}$ and the Hamiltonian map $F_{p_3}(0, x_0, \lambda(0, x_0, \xi), \xi)$ of $p_3$ has non-zero real eigenvalues $\pm \mu(x_0, \xi)$ for $\xi \in \R^n \setminus \{0\}.$\\

$(H_2)$ If $(H_1)$ holds for $(0, x_0, \xi)$, then there exists an open neighborhood $U_{x_0}$ of $x_0$ such that  $\Delta_1(0, x, \xi) = (\pa_t \Delta_1) (0, x, \xi) = 0$ for $x \in U_{x_0}$ and $\xi \in \R^n \setminus \{0\}.$\\

Our main result is the following
\begin{thm} 
Let $x_0 \in U$ be a point for which the hypothesis $(H_0)- (H_2)$ are satisfied. Then there exists a neighbourhood $V_{x_0} \subset U_{x_0}$ of $x_0$ such that for  $T > 0$ sufficiently small,  the operator $P$ is strongly hyperbolic in $\{(t, x);\: 0 \leq t \leq T,\: x \in V_{x_0}\}.$ 
\end{thm}
 The technique of the energy estimates can be modified to cover the case when $(H_1)$ holds only in a conic neighbourhood of a point $\rho_0 = (0, x_0, \xi_0)$. However, to obtain a (WPC) in a neighbourhood of $x_0$, it is necessary to have more sophisticated tools to deduce (WPC) from the microlocal a priori estimates with loss of regularity. As we mentioned above, the assumption $(H_2)$ is valid at all triple characteristic points.
Notice that if $T$ is not small and if for $0 < \delta \leq t < T$ with sufficiently small $\delta$ the operator
$P$ is effectively hyperbolic with double characteristics at some points, we can
combine our result with those of \cite{Iw1}, \cite{Iw2}, \cite{N1},
\cite{N2} to obtain a strongly hyperbolic operator in $\{0 \leq t <
T,\: x \in U\}.$ Our arguments work also if 
we assume that $q_1, q_2, q_3$ are pseudodifferential operators with
real-valued symbols.

\bigskip

 To prove Theorem 1.2, we establish a theorem of existence and
 uniqueness for solutions of the Cauchy problem (see Theorem 8.3). To
 this purpose we obtain an a priori estimate with a loss of regularity
 of order $2N/3-2$ for the operator $P$ near a triple characteristic point.
 We choose $N = \fr{13}{2} \Pi + N_0,$ with $N_0$ an integer and
\begin{equation} 
\label{eq:l.6}
\Pi = \frac{2}{3}  + \sup_{x\in \bar{U}_{x_0}, |\xi| = 1}  | p_2'(0, x, \lambda(0, x, \xi), \xi) (\mu( x,\xi))^{-1}|,
\end{equation}
where $p_2'(t, x, \tau, \xi)$ is the subprincipal symbol of $P$, $\lambda(0, x, \xi)$ denotes the triple root of $p_3 = 0$ and $\mu(x, \xi)$ is the non-vanishing eigenvalue of $F_{p_3}(0, x, \lambda(0, x, \xi), \xi)$. 
Moreover, the integer $N_0$ depends only on $p_3(0, x, \lambda(0, x, \xi), \xi),\: x \in U_{x_0},$ but we are not going to precise
the optimal value of $N_0$. It seems that with a more complicated
analysis of the contribution of the subprincipal symbol $p_2'$
it should be possible to obtain a loss of regularity $2N/3 - 2$ with
$N = \frac{3}{2}\Pi + N_0$ and this is an interesting open problem.

\bigskip
 
We may compare the number (\ref{eq:l.6}) with the loss of regularity for second
order strongly hyperbolic operators $L$ with principal symbol $l_2(t,
x, \tau, \xi)$ given by 
$$2 + \Bigl[ \frac{1}{2} + \sup_{\rho  \in \Sigma_1} |l_1'(\rho) (\mu(\rho)^{-1}|\Bigr],$$
where $[z]$ is the integer part of $z$ (see  \cite{O}, \cite{I}, \cite{M} for a special class of operators with double characteristics and \cite{N2} for the general case). Here 
$$\Sigma_1 = \{\rho = (t, x, \tau, \xi)\in G \times \R^{n + 1}\setminus \{0\}:\:l_2(\rho) = 0, d l_2(\rho) = 0\}$$
is the double characteristic set of $L$, $l_1'(\rho)$ is the subprincipal symbol of $L$ and $\mu(\rho)$ is the non-vanishing eigenvalue of the Hamiltonian map $F_{l_2}(\rho)$ at  $\rho \in \Sigma_1$.  It is important to note that the loss of regularity $M$ for the solutions of the Cauchy problem for $P$ is {\bf bounded} from below by
\begin{equation} \label{eq:1.7}
\sup_{x \in \bar{U}, \: |\xi| = 1} | \im (p_2'(0, x, \lambda(0, x, \xi), \xi))(\mu(x,\xi))^{-1}| \leq 2 n(M + 3).
\end{equation}
This follows from the necessary condition (36) in Theorem 3 in \cite{IP}. Thus our result with $M = 2N/3 - 2$ is compatible with this lower bound.

\subsection{Comments on the proof of the main result}

The proof of Theorem 1.2 is long and technical. It is based on the energy estimates obtained in
Theorems 8.1 and 8.2 and, as we mentioned above, we cannot avoid the
loss of regularity which is related to the ratio of the
subprincipal symbol and the non-vanishing eigenvalue of the
Hamiltonian map. This is one of the main differences compared to the hyperbolic 
operators of principal type (see for example the analysis in Section
23.4 in \cite{h3}.)

\bigskip

First by a change of variables $(t, x)$ we reduce the analysis to the
case when the principal symbol  $p_3$ in the new variables, which we denote again by $(t, x)$,  has the form 
$$
p_3(t, x, \tau, \xi) = \tau^3  -( t a_2(t, x,
\xi) + \alpha (x, \xi))\tau + t^2b_3(t, x, \xi),
$$
where $a_2, \alpha$, are real-valued symbol homogeneous of order 2 with respect to $\xi$ and $a_2(t, x, \xi)$  $ \geq c
|\xi|^2,\:c > 0,\: \alpha(x, \xi) \geq 0$, while $b_3(t, x, \xi)$ is a real-valued symbol of ordre 3 in $\xi$ (see Section 2). Here we use essentially the condition $(H_2)$ to present the term with third order derivaties in $x$ as $t^2 b_3(t, x, D_x)$. If $(H_2)$ is not satisfied, we will have a term
$$a_3(t, x, \xi) = t^2 b_3(t, x, \xi) + t c_3(x, \xi) + \beta(x, \xi)$$
with third order symbols $b_3, c_3, \beta$. Of course, the hyperbolicity of $p_3$ implies that
$$ |a_3|^2 \leq \fr{4}{27} (t a_2 + \alpha)^3,\: t \geq 0,$$
but it is quite difficult to exploit this condition working with the energy ${\mathscr E}_N(u)$ and the time function $f = \fr{t}{3} + \la \xi \ra ^{-2/3}$ defined below.  A different  
choice of $f$ adapted to the structure of $p_3$ and some more complicated energy techniques seem to be necessary to cover the general case when $c_3 \neq 0$ and $\beta \neq 0.$
The reader may consult the work of Nishitani \cite{N2}, where the choice of $f$ is related to a microlocal model of the principal symbol with double characteristics. 

Next,  we introduce, in Section 4, the scaling $t =
\ep^{2/3}s, x = \ep y,\: \ep > 0,$ and we transform our operator into
the operator ${\mathcal P}$ with respect to $(s, y)$ (see Section 4 for
the notations).  Since we are 
interested in showing that the Cauchy problem is well posed for
sufficiently small $t$ and since $P$ is strictly hyperbolic for $t$
positive and small enough, we can investigate the operator ${\mathcal P}$. We denote below again by $(t, x)$ the new variables. The
symbols $a_2(t, x, \xi),\: \alpha(x, \xi)$ are transformed to symbols 
$$a_2^{\ep} = a_2(\ep^{2/3}t, \ep x, \xi),\: \alpha^{\ep} = \alpha(\ep x, \xi)$$
and this is important for the pseudodifferential calculus developed in
Sections 3-4. Eventually we choose $\ep = {\mathcal O} (N^{-1})$, where $N = \frac{13}{2}\Pi + N_0$
 is a large integer, so that  $0 < \epsilon_0 \leq \ep N \ll 1.$

\bigskip

The so called {\it time function} $f(t, \xi) = \frac{t}{3} + \la \xi \ra ^{-2/3}$
plays an important role in the calculus of pseudodifferential
operators with order function $m_N^t = f^{-N}(t, \xi)$ and metric 
$$g_{(x, \xi)}^{\ep} = \ep^2 |dx|^2 + \la \xi \ra^{-2} |d\xi|^2.$$
The reason of this choice is that the commutator $[ a_2^{\ep}(t, x,
D_x), f^{-N}(t, D_x)]$ has symbol in the space $S(f^{-N} \ep N \la \xi \ra,
g^{\ep})$ so $[a_2^{\ep}(t, x, D_x), f^{-N}(t, D_x)] f^{N}(t, D_x)$ becomes a first order operator whose norm is not
depending on $N$ and, in particular, on $\Pi.$ This is proved in
Proposition 5.2.

\medskip

The main idea is to multiply $f^{-2N}(t, D_x){\mathcal P}u$ by the {\it multiplier}
$$M u = \psi(t) \Bigl[D_t^2 - \theta (t a_2^{\ep}(t, x, D_x) u + \ae(x, D_x) u)\Bigr],$$
where we choose $\theta = 1/3$ and $\psi(t) = \fr{\el}{t},\: \lambda > 0.$ For the analysis of the problem with data $D_t^j u(T, x) = 0, \: j = 0, 1,2$, we use the function $\varphi(t) = t \el$.
We study the expression 
\begin{equation} \label{eq:1.7bis}
- 2 \im \la f^{-2N}(t, D_x) {\mathcal P}u, Mu \ra,
\end{equation}
$\la  ,  \ra$ denoting the scalar product in $L^2(\R^n)$
and $u \in C_0^{\infty} ([0, T] \times U)$ has traces $D_t^j u(t_0, x) = 0, \: j = 0,1,2, \: 0 \leq t_0 \leq t \leq  T.$ The last condition guarantees that $Mu$ is well defined for $t_0 = 0.$ The above
expression modulo lower order terms is a sum of
15 terms and we make a quite detailed analysis of all these terms in
Section 5. The purpose is to find, by integration by parts, ``positive''
terms with a big coefficient of order ${\mathcal O}(N)$ which will absorb in the
energy estimate the contributions with "indefinite" (possibly negative) sign.

In fact, we have many indefinite terms, while the positive ones come from the expression
$$\partial_t \mathscr{E}_N(u) + \fr{2N}{3} \mathscr{E}_{N+ 1/2}(u) + 2\lambda {\mathscr E}_N(u),$$
where
\begin{eqnarray*}
\mathscr{E}_N(u)  = \psi \Bigl[\| f^{-N} u''\|_0^{2}    
+\fr{2}{3} \re \la f^{-2N} (t \aee + \ep^{-2/3} \ae) u', u' \ra \nonumber \\
 +\fr{1}{3} \|f^{-N} (t \aee u + \ep^{-2/3} \ae ) u\|_0^2 
+ \fr{2}{3} \re \la f^{-2N} u'', (t \aee + \ep^{-2/3} \ae) u\ra \Bigr]  \nonumber \\                                                
+ \epsilon^{1/3} t \el 2 \im \la f^{-2N} b_{3}^{\epsilon} u,
u'\ra . 
\end{eqnarray*}

The quantity in the third line above, involving powers of $\ep$ and
$t$, is a lower order perturbation. We introduce the {\bf energy} of
order $k \in N$ by the expression

\def\a0{{\bf a}_2^{\ep}}
\begin{eqnarray} \label{eq:1.11}
E_{k}(u) = \psi\Bigl(\frac{1}{3} \| f^{-k} u'' \|_0^{2} +\frac{2}{3}  \re \la f^{-k}(t \a0 + \ep^{-2/3} \ae) u' ,
f^{-k}u' \ra \nonumber \\
+ \frac{1}{6} \| f^{-k}(t \a0 u + \ep^{-2/3} \ae) u \|_0^{2} + \frac{2}{3} \|f^{-k}( u'' +\frac{1}{2} (t \a0 + \ep^{-2/3} \ae)u)\|_0^2\Bigr),
\end{eqnarray}
where $\a0$ is the operator with symbol $a_2^{\ep}(0, x, \xi)$ and $\|.\|_k$ denotes the $H^k(\R^n)$ norms. 
Next we prove that modulo lower order terms we have $\mathscr{E}_k(u) \geq E_k(u).$
To gain control on the terms involving $\|f^{-k}u(t, .)\|_j$ norms with $j = 1, 2$ and $k = N, N + 1/2$, we would like to
exploit the terms with $H^2(\R^n)$ norms having large coefficients. However, we have positive terms only with
the norm $t\|f^{-N- 1/2} u\|_2^2$ and, as $t$ becomes close to 0, we cannot absorb the $H_2$ norm of $f^{-N- 1/2}u.$ This is the principal difficulty when we try to absorb the lower order terms created by $\im \la f^{-2N} b_2^{\ep} u, \aee u\ra$ and the second order operator $b_2$.
 A way around this is to use the key inequality 
$$
\la \xi \ra^4 f^{-k}\leq t \la \xi \ra^4 f^{-k-1} + \la \xi \ra ^2 f^{-k-3},
$$
established in Lemma 6.1. Thus, for example, we have an estimate
$$
\|f^{-N + 1/2} u\|_2^2 \leq t \|f^{-N}u \|_2^2 + \|f^{-N -1}
u\|_1^2.
$$
Applying the above equality twice, we get also norms of the form $\|f^{-N
  - 5/2} u\|_0$. Following this way, we have  
negative terms with coefficients $t$, but we have now
generated other negative terms involving norms with {\it weights}
$f^{-N - 3/2}$ and $f^{-N - 5/2}$ and these also have to be absorbed. 
However, the latter types of terms are not included in our {\bf energy}
expressions $E_{N+ 1/2}(u), E_N(u)$ given by (\ref{eq:1.11}).

To deal with them, we apply to
just a fraction of the positive terms in $\psi E_{N + 1/2}(u)$, having a large
coefficient proportional to $N$, the inequality
\begin{eqnarray} 
\label{eq:6.10b}
\psi \|\fNt u'' \|^2&  \geq & \pa_t \Bigl( \psi \|f^{N -1}  u'\|^2 \Bigr) - \psi'\|f^{-N -1} u'\|^2 
 + (2N/3 - 4/3) \psi \|f^{-N - 3/2} u'\|^2 \nonumber \\
&  &
+ \pa_t \Bigl(
 \psi \|f^{-N - 2} u \|^2\Bigr) - \psi'\|f^{-N - 2} u \|^2 + \frac{2N
   + 1}{3} \psi \|f^{N - 5/2} u\|^2.
\end{eqnarray}
A similar inequality for $\psi \|f^{-N - 1} u\|_1^2$
completes our technical toolkit. 

Finally, in Section 7 we show that we can
absorb all non positive terms. In Section 8 we get the energy estimates
with loss of $2N/3 - 2$ derivatives in $x$ which imply, by a standard
argument, the well posedeness of the Cauchy problem.\\

\section{Hyperbolic operators with triple characteristics}
%
%
\setcounter{equation}{0}
\setcounter{thm}{0}
\setcounter{prop}{0}  
\setcounter{lem}{0}
\setcounter{cor}{0} 
\setcounter{defn}{0}
\setcounter{rem}{0}
%

 In this section we use the notations of Section 1. First we will change the variables so that in the new variables  the
principal symbol has not a term involving $\tau^2.$  Let us write
$$p_3 = \Bigl(\tau + \frac{1}{3} q_1\Bigr)^3 + (q_2 - \frac{q_1^2}{3})\Bigl(\tau + \frac{1}{3}q_1\Bigr) + q_3 - \frac{q_1 q_2}{3} + \frac{2 q_1^3}{27}.$$
Then the term without coefficient $(\tau + \fr{1}{3} q_1)$ is just $ \fr{\Delta_1}{27}.$
With a change of variables 
$$s = t, \: y_j = f_j(t, x),\: j=1,...,n,$$
we may transform for small $t \geq 0$ the symbol $\tau + \frac{1}{3} q_1$ into $\sigma$. Let $\frac{1}{3}q_1 = \sum_{j = 1}^n \alpha_j(t, x) \xi_j.$
It is sufficient to solve the
first order hyperbolic equations
$$ \frac{\partial f_j}{\partial t} +\sum_{k=1}^n \alpha_k(t, x) \frac{\partial f_j}{\partial x_k} = 0, \: j=1,...,n$$
with initial data $f_j(0, x) = x_j$. The Jacobian $ J = \det \frac{D(s, y)}{D(t, x)}$ is different from 0 for small $t \geq 0$ and $x \in \bar{U}$ and we can solve the system for $f_j, j = 1,...,n,$ for small $|t| \leq \eta.$ 

Now by using the same notations $(t, x)$ for the new variables, consider the operator with principal symbol
\begin{equation} \label{eq:2.1}
{\bf p}_3 =  \tau^3 - r_2(t, x, \xi) \tau + r_3(t, x, \xi),
\end{equation}
where $r_j$ are real-valued symbols homogeneous with respect to $\xi$ of order $j = 2, 3.$ Clearly, the symbol $r_3$ is just $\fr{\Delta_1}{27}.$ If we have a multiple real root $\tau = \lambda(0, x_0, \xi)$ of ${\bf p}_3 = 0$ for $(0, x_0, \xi)$,
then $\lambda(0, x_0, \xi)$ is a root of $3\tau^2 - r_2(t, x, \xi) = 0$. Thus $r_2(0, x_0, \xi) \geq 0$. Moreover, the root $\lambda$ is triple if and only if $\lambda(0, x_0, \xi) = r_2(0, x_0, \xi) = 0.$ If $\tau = 0$ is a triple root at $(0, x_0, \xi)$ and if the fundamental matrix of ${\bf p}_3$ at $(0, x_0,0, \xi)$ has non zero real eigenvalues, then $\partial_t r_2(0, x_0, \xi) > 0$ (see Lemma 8.1  in \cite{IP}), On the other hand, the hyperbolicity
of the operator implies that the discriminant $\Delta$ satisfies the inequality
$$\Delta = 4 r_2^3 - 27 r_3^2 \geq 0.$$
If at a point $(0, x_0, \xi_0)$ we have $r_2(0, x_0, \xi_0) < 0,$ then $\Delta(0, x_0, \xi_0) < 0$ and we will have two complex conjugated roots. Thus in a neighborhood of $(0, x_0)$ we have
$$ r_2(t, x , \xi) = \alpha( x, \xi) + t a_2(t, x, \xi)$$
with 
$$ a_2(0, x, \xi) \geq c|\xi|^2, \: c > 0,\: \alpha(x, \xi) \geq 0.$$

In the new variables, denoted by $(t, x)$, the operator $P$ is transformed into
\begin{eqnarray}
\label{eq:2.2}
{\mathcal P} & = & D_t^3  -(t a_2(t, x, D_x) + \alpha(x, D_x))D_t +  a_3(t, x, D_x)
+ b_2(t, x, D_x)\nonumber 
\\
&  &
 +  b_1( t, x, D_x)D_t +  b_0(t, x)D_t^2 + c_1(t, x, D_x) + c_0(t, x)D_t + d_0(t, x).
\end{eqnarray}
Here $ a_2, a_3$ are homogeneous polynomials with respect
to $\xi$ respectively of order 2, 3,  $b_j(t,x, \xi)$ are
homogeneous polynomials with respect to $\xi$ of order $j$, while
$b_0, c_0, d_0$ are smooth functions and $c_1(t, x, \xi)$ is a first
order differential operator with respect to the variable $x$.  Without loss of generality we may assume that
$a_2(t, x, D_x)$ and $\alpha(x, D_x)$ are self-adjoint positive operators. This will change the operators $b_1(t, x, D_x)D_t,\:c_0(t, x)D_t$ which is not important for our argument.
According to the condition $(H_2)$, in a neighborhood $U_{x_0}$ of $x_0$ $a_3$ has the form
$$a_3 (t, x, \xi) = t^2 b_3(t, x, D_x)$$
with a symbol $b_3(t, x, \xi)$ homogeneous of order 3 in $\xi$. We may also assume that $b_3(t, x, D_x)$ is a self-adjoint operator changing the lower order terms. Since we have a coefficient $t^2$, this is not important for the subprincipal symbol $p_2'$ at $(0, x, \xi).$

Next, throughout our exposition we will assume that the
principal symbol $p_3(t, x, \tau, \xi)$
has a triple root $\tau = 0$ for $t = 0, x_0 \in U, \forall \xi \in \R^n \setminus \{0\}$
and we examine the operator having the form (\ref{eq:2.2}) with $\alpha(x_0, \xi)=0.$ We suppose in the following that $x_0
= 0.$ Since 
$\alpha(x, \xi) \geq 0,$ in a neighborhood of $0$, we have
\begin{equation} \label{eq:2.3}
\alpha (x, \xi) = \sum_{i, j = 1}^n x_i x_j f_{i, j} (x,\xi).
\end{equation}
Therefore the Hamilton map $F_{p_3}$ of $p_3$ for $t = 0, x = x_0, \tau = 0$ has
non-vanishing eigenvalues if only if $a_2(0, x_0, \xi) \neq 0, \xi \in \R^n \setminus \{0\}.$

\bigskip

%
It is well known that the subprincipal symbol and the eigenvalues of
the Hamilton map are invariant on the characteristic points $\rho \in
\Sigma_1 = \{ \rho \in T^*(G) \setminus \{0\}:\: p_3(\rho) = 0, dp_3(\rho) = 0\}$. Thus the number $\Pi$ defined by (1.7) can be expressed by
the subprincipal symbol $p_2'(0, x, \xi)$ and $a_2(0, x, \xi).$

We extend the coefficients of $a_2, \alpha,  b_k, k = 0, 1, 2, 3$ and $c_1,
c_0, d_0$  for $x \in \R^n$ as smooth functions. Thus in the analysis 
in Section 3-8 we will assume that the operator ${\mathcal P}$ is
defined in $\R^n$. Moreover, our arguments work with small modifications if
$a_2, \alpha,  b_k, c_k$, etc, are classical pseudodifferential operators with
symbols 
$$a_2(t, x, \xi) \in S^2_{1, 0}(\R^{n+1} \times \R^n), b_k(t, x, \xi)
\in S^k_{1, 0}(\R^{n+1} \times \R^n), c_k(t, x, \xi) \in S^k_{1,
  0}(\R^{n+1} \times \R^n)$$ 
 depending smoothly on the parameter $t$.\\

\section{Some classes of symbols}
%
%
\setcounter{equation}{0}
\setcounter{thm}{0}
\setcounter{prop}{0}  
\setcounter{lem}{0}
\setcounter{cor}{0} 
\setcounter{defn}{0}
\setcounter{rem}{0}
%

Let
\begin{equation}
\label{3}
f(t, \xi) = \frac{t}{3} + \langle \xi \rangle^{-2/3},
\end{equation}
where  
$\langle \xi \rangle^{2} = 1 + |\xi|^{2}.$ Then clearly $ f $ is a symbol in the class $ S^{0}_{1,
  2/3} $, when derivatives with respect to $ t $ are considered, but it
is in the class $ S^{0}_{1, 0} $ if $ t $ is just a parameter and no
derivatives with respect to $ t $ are involved.

It will be convenient for us to use the Weyl calculus formalism, in
the variables $ x $, in order to establish an \textit{a priori}
estimate for the operator we deal with. From now on $ t $ will be
regarded as a non-negative parameter.

Let $ \epsilon > 0 $ be a small positive number. We consider the
metric in $ T^{*}\R^{n} $ defined by
\begin{equation}
\label{eq:10metric}
g_{(x, \xi)}^{\epsilon} = \epsilon^{2} |dx|^{2} + \langle \xi \rangle^{-2} |d\xi|^{2}
\end{equation}
which is almost  the classical $ (1, 0)-$metric. 
In the following we will write $ g $, when there is no ambiguity. It is well known that $g$ is
a slowly varying metric. 

Let $ N $ be a positive integer. In what follows the size of $ N $ is
determined in terms of the problem.
We define the function
\begin{equation}
\label{eq:orderfctn}
m_{N}^{t}(\xi) = f^{-N}(t, \xi)
\end{equation}
and it is trivial to verify that $ m_{N}^{t} $ is an order function. Then we may define the classes $ S(m_{N}^{t}, g^{\epsilon}) $  of
symbols in the standard way.

We point out explicitly that $ t $ is just a parameter at this level
and that if there is no ambiguity we may omit it in our notation. We have the following
\begin{prop}
\label{prop:f-N}
$ f^{-N}(t, \xi) \in S(m_{N}^{t}, g^{\epsilon}) $.
\end{prop}
\begin{proof}
Of course we must check only $ \xi $-derivatives. We have that 
$$ 
\partial_{\xi_{j}} f^{-N}(t, \xi) = -N f^{-N}(t, \xi) \frac{\langle
  \xi \rangle^{-2/3}}{f(t, \xi)} \left( - \frac{2}{3}\right)
\frac{\xi_{j}}{\langle \xi \rangle} \langle \xi \rangle^{-1}.
$$
Hence we have the estimate $ | \partial_{\xi_{j}} f^{-N}(t, \xi) |
\leq C_{N} f^{-N}(t, \xi) \langle \xi \rangle^{-1} $. A simple
iteration concludes the proof.
\end{proof}
\begin{rem} \label{fN}
In particular we deduce that 
$$ 
\partial_{\xi}^{\alpha} f^{-N}(t, \xi) = {\mathcal O}\left(
  N^{|\alpha|} f^{-N}(t, \xi) \langle \xi \rangle^{-|\alpha|}\right).
$$
\end{rem}
Given a symbol  $ a(t, x, \xi) \in S(m^{t}_{N}, g^{\epsilon}) $, which
we may also denote by $ a^{t} (x, \xi) $, the Weyl pseudodifferential
operator associated with it is defined by the formula
$$ 
a^{t w}u (x) = (2 \pi)^{-n} \int\int e^{i \langle x - y, \xi \rangle}
a^{t}\left( \frac{x+y}{2}, \xi\right) u(y) dy \ d\xi.
$$
We recall the composition rule for two symbols (see
e.g. \cite{h3}). Define 
$$ 
g^{\sigma}(w) = \sup_{w'} \frac{| \sigma( w, w') |^{2}}{g(w')}.
$$
Here $ \sigma $ denotes the symplectic form defined on $ T(T^{*}\R^{n})
\times T(T^{*}\R^{n}) $, which in our local coordinates is given by
$$ 
\sigma\left((x, \xi), (y, \eta)\right) = \langle y, \xi \rangle -
\langle x, \eta \rangle.
$$
\begin{thm}[Theorem 18.5.4, \cite{h3}]
\label{wcomp}
Let $ g $ be a temperate metric with $ g \leq g^{\sigma} $ and let $
m_{1} $, $ m_{2} $ be $( \sigma $, $ g )$ temperate order functions. Let
$ a_{j} \in S(m_{j}, g),\: j = 1,2 $. Then the composition of the associated
pseudodifferential operators is associated to a symbol map $ (a_{1},
a_{2}) \mapsto a = a_{1} \# a_{2} $ from $ S(m_{1}, g) \times S(m_{2},
g)$ to $ S(m_{1} m_{2}, g) $ and $ a $ is defined by
\begin{equation}
\label{formcomp}
a (x, \xi) = \exp\left( \frac{i}{2} \sigma(D_{x}, D_{\xi}; D_{y},
  D_{\eta}) \right) a_{1} (x, \xi) a_{2} (y, \eta)_{\big |_{(x, \xi) =
  (y, \eta)}}.
\end{equation}
Let
\begin{equation}
\label{h}
h (x, \xi)^{2} = \sup_{w} \frac{g_{x, \xi}(w)}{g^{\sigma}_{x, \xi}(w)},
\end{equation}
then we have that for every integer $ M $ the map associating $ a_{1}
$, $ a_{2} $ to the remainder term
$$ 
a_{1}\# a_{2}(x, \xi) - \sum_{j <M} \frac{ \left(i \sigma(D_{x}, D_{\xi};
  D_{y}, D_{\eta}) \right)^{j}}{2^{j} j!} a_{1}(x, \xi) a_{2}(y, \eta) 
$$
evaluated on the diagonal $ (x, \xi) = (y, \eta) $, is continuous with
values in $ S(h^{M} m_{1} m_{2} , g) $.
\end{thm}
\begin{rem}
\label{nox}
We explicitly note that the above formula (\ref{formcomp}) reduces to
the usual symbol composition formula (i.e. with no effect of the Weyl
operator definition) if $ a_{1} $ or $ a_{2} $ does not depend on $ x
$; thus (\ref{formcomp}) reduces to
$$ 
a (x, \xi) = \sum_{|\alpha| \geq 0}
\frac{1}{\alpha!} \partial_{\xi}^{\alpha} a_{1} (\xi) 
D_{x}^{\alpha} a_{2} (x, \xi),
$$
or the analogous symmetric formula if $ a_{2} $ is independent of $ x
$. 

Note that a different way of writing the above formula is
$$ 
a (x, \xi) = \exp\left( \frac{i}{2} \langle D_{y}, D_{\xi} \rangle
\right) a_{1} (\xi) a_{2} (y, \eta)_{\big |_{(x, \xi) =
  (y, \eta)}}.
$$
\end{rem}
\begin{rem}
\label{gepsilon}
An easy and explicit calculation yields that
\begin{equation}
\label{ges}
(g^{\epsilon}_{x, \xi})^{\sigma} = \langle \xi \rangle^{2} | dx|^{2} +
\epsilon^{-2} | d\xi |^{2},
\end{equation}
and consequently the function $ h $ is given by
\begin{equation}
\label{h1}
h (x, \xi) = \frac{\epsilon}{\langle \xi \rangle }.
\end{equation}
Evidently in this case 
$$ 
g^{\epsilon}_{x, \xi} \leq (g^{\epsilon}_{x, \xi})^{\sigma}.
$$
\end{rem}
We need to have a better control of this remainder terms in order to
estimate the composition symbol with respect to the parameter $ N $
introduced above.

We recall a result due to J.--M. Bony, \cite{jmb}, according to which
the composition $ a_{1} \# a_{2} $ is written as a finite sum plus a
remainder term:
\begin{multline}
\label{b1}
a_{1} \# a_{2} (x, \xi) = \sum_{p=0}^{M-1} \frac{1}{p!} \left(
  \frac{i}{2} \sigma(D_{x}, D_{\xi}; D_{y}, D_{\eta}) \right)^{p}
a_{1}(x, \xi) a_{2}(y, \eta)_{\big |_{(x, \xi) = (y, \eta)}} 
\\
+ R_{M}(a_{1}, a_{2})(x, \xi),
\end{multline}
where
\begin{multline}
\label{brem}
R_{M}(a_{1}, a_{2})(x, \xi) = \int_{0}^{1} \frac{(1 -
  \theta)^{M-1}}{(M - 1)!}  
\\
\cdot 
\frac{1}{(\pi \theta)^{2n}} \int\int
e^{-(2i/\theta) \sigma((x, \xi) - 
  (t, \tau), (x, \xi) - (y, \eta))} 
\\
\cdot 
\left( \frac{i}{2} \sigma(D_{t},
  D_{\tau}; D_{y}, D_{\eta}) \right)^{M} a_{1}(t, \tau) a_{2}(y, \eta) 
 dt d\tau dy d\eta d\theta
\end{multline}
\begin{prop}
\label{prop:bony}
Assume that $ a_{i} \in S(m_{i}, g) $, where $ g $ is a slowly
varying, temperate metric such that $ g \leq g^{\sigma} $. Then both $
R_{M} $ and the restriction to the diagonal of 
$$ \sigma((D_{x},
D_{\xi}); (D_{y}, D_{\eta})) a_{1}(x, \xi) a_{2}(y, \eta) 
$$ 
belong to
$ S(m_{1} m_{2} h^{M}, g) $.
\end{prop}

\section{Scaling and multiplier}
%
%
\setcounter{equation}{0}
\setcounter{thm}{0}
\setcounter{prop}{0}  
\setcounter{lem}{0}
\setcounter{cor}{0} 
\setcounter{defn}{0}
\setcounter{rem}{0}
%

To obtain an a priori estimate and to deal with the lower order terms
we introduce a scaling 
\begin{equation}
\label{eq:pepsilon}
t = \epsilon^{2/3}s, \qquad x = \epsilon y, \qquad \epsilon > 0.
\end{equation}
Multiplying by $\epsilon^2$, we obtain an operator
\begin{eqnarray}
\label{eq:defPwithepsilon}
{\mathcal P} = D_s^3 - s a_2(\epsilon^{2/3} s, \epsilon y, D_y)D_s +
b_2(\epsilon^{2/3}s, \epsilon y, D_y)\nonumber 
 - \ep^{-2/3} \alpha(\ep y, D_y)D_s  \nonumber \\
 + \epsilon^{1/3} \Bigl[  s^2b_3(\epsilon^{2/3}s, \epsilon y, D_y) + 
b_1(\epsilon^{2/3} s, \epsilon y, D_y)D_s\Bigr] + \epsilon^{2/3} b_0(\epsilon^{2/3}s, \epsilon y)D_s^2 \nonumber \\
+\epsilon c_1(\epsilon^{2/3}s, \epsilon y, D_y) + \ep^{4/3} c_0(\ep^{2/3}s, \ep y)D_s + \ep^2 d_0(\ep^{2/3}s, \ep y).
\end{eqnarray}
Here we applied (\ref{eq:2.2}) and we use for simplicity the same notation ${\mathcal P}$ for the transformed operator. Moreover, $a_2, \: \alpha, b_k, k= 0, 1, 2, 3$, etc. are the symbols of Section 2.

Since we are interested in obtaining an estimate for $0 \leq t_0 \leq t \leq T$ with initial conditions on $t_0 = 0$,
with sufficiently small $T > 0$, we may think of $\epsilon$ as a
parameter which is going to be chosen sufficiently small; actually it
will be fixed below as 
 $\epsilon = {\mathcal O}(\frac{1}{N})$, where $N = \fr{13}{2} \Pi + N_0$ and $\Pi$ was defined in the Introduction.

We may also return to the notation $ (t, x) $ for the time and space variables respectively, without any risk
of misunderstanding.  Let
\begin{equation}
\label{eq:redop}
P_0 = D_t^3 - t a_2(0, \epsilon x, D_x)D_t +
b_2(0, \epsilon x, D_x)
\end{equation}
be the leading term in ${\mathcal P}$ having no factors depending on $\epsilon.$

It is convenient to use the following notation:
\begin{equation}
\label{eq:epsdefa}
a_{2}^{\epsilon}(t, x, D_{x}) = a_{2}(\epsilon^{2/3} t, \epsilon x, D_{x}),\: \alpha^{\ep} = \alpha(\ep x, D_x),
\end{equation}
and
\begin{equation}
\label{eq:epsdefb}
b_{j}^{\epsilon}(t, x, D_{x}) = b_{j}(\epsilon^{2/3} t, \epsilon x, D_{x}),\: c_j^{\ep}(t, x, D_x) = c_j(\ep^{2/3} t, \ep x, D_x)
\end{equation}
for the differential operators appearing in the
definition (4.2) of ${\mathcal P} $, emphasizing the dependence on
the parameter $ \epsilon $.

 For $ u $, $ v \in C_{0}^{\infty}(\overline{\R^{+}} \times 
\R^n) $, we denote by
$$ 
\langle u , v \rangle = \int_{\Omega} u(t, x) \bar{v}(t, x) dx, 
$$
the usual scalar product in $ L^{2}(\R^n) $ w.r.t. the space variables $x$. Also we denote by $\|v(t,.)\|_k$ the norms in the spaces
$H^k(\R^n).$

\bigskip

In order to deduce an energy estimate, we need a second order multiplier
operator. In what follows we use the multiplier
\begin{equation}
\label{eq:multiplier}
M(t, x, D_{t}, D_{x}) = \psi(t)\Bigl (D_{t}^{2} - \theta t a_{2}^{\epsilon}(t, x, D_{x}) - \theta\ep^{-2/3}\alpha^{\ep}(x, D_x)\Bigr),
\end{equation}
where  $ \theta$ denote a positive constant to be chosen later and $\psi(t) = \frac{\el}{t}, \: \lambda > 0.$  Clearly, we have the inequalities
\begin{equation} \label{psi}
-\psi'(t) > \lambda \psi(t) ,\: -\psi'(t) > \frac{\el}{t^2} > \frac{e^{-4\lambda t}}{t^2} = \psi^2(t).
\end{equation}

 For $ u \in C_{0}^{\infty}(\overline{\R^{+}} \times \R^n)$ and $0 \leq t_0 \leq t \leq T$
 we compute the expression 
$$ 
- 2 \im \langle f^{-2N}(t, D_{x}) \mathcal{P} u , M u \rangle. 
$$
Here $ f^{-2N}(t, D_{x}) $ denotes the pseudodifferential operator whose
symbol is $ f^{-2N}(t, \xi) \in S(m_{2N}^{t}, g^{\epsilon}) $. We suppose in addition that $u(t_0, x) = u_t(t_0, x) = u_{tt}(t_0, x) = 0$. Thus in the case $t_0 = 0$ the terms with $\psi(t)$ have a sense for $t=0$.

We have
\begin{align}
\label{eq:4.8}
- 2 \im \langle f^{-2N}(t, D_{x}){\mathcal  P}u , Mu \rangle 
&=  
2 \re \langle \psi f^{-2N} \left( \partial_{t}^{3} + t a_{2}^{\ep}
  \pa_{t} \right) u, \left(\pa_{t}^{2} + \theta t a_{2}^{\ep} \right)
u \ra 
\notag \\
&\phantom{=}
+ 2\re\theta \ep^{-4/3}\la \psi f^{-2N}  \ae \pa_t u,  \ae u \ra 
\notag \\
&\phantom{=}
+ 2 \theta \ep^{-2/3}\re\la \psi f^{-2N}\left(\partial_{t}^{3} + t
  a_{2}^{\ep} \pa_{t} \right) u,  \ae u \ra 
\notag \\
&\phantom{=}
+ 2\re\ep^{-2/3}\la \psi
f^{-2N} \ae \pa_t u, \left(\partial_{t}^{2} + \theta t a_{2}^{\ep}
\right) u \ra  
\notag
\\
&\phantom{=}
 + 2 \im \langle \psi f^{-2N} \epsilon^{1/3} t^{2} b_{3}^{\epsilon} u, 
\left(\partial_{t}^{2} + \theta t a_{2}^{\ep}  + \theta \ep^{-2/3} \ae \right) u \rangle
\notag  \\
&\phantom{=}
 + 2 \im \langle \psi f^{-2N} b_{2}^{\epsilon}u, \left( \pa_t^2 +
   \theta  t a_2^{\ep} + \theta \ep^{-2/3} \ae \right)  u \rangle + \text{lower order terms}
\notag 
\\
&= 
\sum_{j = 1}^{5}I_j + \sum_{j = 1}^4 J_j + \sum_{k= 1}^3 A_k +
\sum_{\nu= 1}^3 B_{\nu} + \text{lower order terms}. 
\end{align}
Here $I_j$, $J_j$ denote the terms arising from the scalar product
with the operator $D_t^3 - (t a_2^{\ep} + \ep^{-2/3} \ae)D_t $, the
$A_k$ come from the third order operator w.r.t. $D_x$ and finally the $B_{\nu}$
originate from the lower order term $b_2^{\ep}.$ 
Moreover, we denoted by ``lower order terms'' the terms of order 1 or 2
involving the operators $b_1, d_0, d_1$. It will be evident after the
discussion below that they do not have any influence whatsoever on the
energy estimate for $ \mathcal{P} $ that we are going to deduce and
hence, to avoid burdening the exposition with useless details we omit
a discussion of those terms.

In the next section we are going to estimate each term $I_j$ with the
purpose of putting in evidence a positive energy containing the weight
$ f^{-N} $ as well as $ f^{-N-1/2} $.

\section{Estimate of the terms in (\ref{eq:4.8})}
%
%
\setcounter{equation}{0}
\setcounter{thm}{0}
\setcounter{prop}{0}  
\setcounter{lem}{0}
\setcounter{cor}{0} 
\setcounter{defn}{0}
\setcounter{rem}{0}
%

\subsection{Estimate of $ I_{1} $}

For the term $I_1$ we have

\begin{eqnarray} 
\label{5}
- 2 \im \la \psi \fN D_t^3 u , D_t^2 u) & = & 
\nonumber
\\
2 \re \la \psi \fN u''', u'' \ra 
& = & 
\pa_t\Bigl( \psi \|f^{-N} u''\|^2\Bigr) + 2N/3\psi \|f^{-N -
  1/2}u''\|^2 - \psi' \|f^{-N}u'' \|^2. 
\end{eqnarray}
Here we write, as we did in the preceding section, $ f^{-2N} $
instead of $ \op(f^{-2N}) $, for the sake of simplicity. 

Note that we have used the fact that $ f $, or rather its powers, is self
adjoint as an operator w.r.t. the $ x $ variables when acting on
smooth functions with compact support.

\subsection{Estimate of $I_{3}$ and $J_3$}

Due to Proposition \ref{prop:f-N}, we have that, as a symbol, $
f^{-2N} \in S(m^{t}_{2N}, g^{\epsilon})  $, where $ \epsilon $ is a
positive parameter to be chosen below and the variable $ t $ can be
regarded, for the time being, as a parameter. The order function $
m^{t}_{2N} $ has been defined in (\ref{eq:orderfctn}).

Taking into account that we performed a dilation by $ \epsilon$, we conclude that
\begin{prop}
\label{prop:a2epsilon}
The symbols $ a_{2}^{\epsilon},\: \ae,\: b_{2}^{\epsilon} $ belong to $ S(\langle \xi
\rangle^{2}, g^{\epsilon}) $, as symbols in the $ x $ variables. It is
then straightforward to show that actually
\begin{equation}
\label{eq:a-and-b}
\epsilon^{-\frac{2}{3} j} \partial_{t}^{j} a(t, x, \xi) \in
S(\langle\xi\rangle^{2}, g^{\epsilon}),
\end{equation}
where $ a $ denotes $ a_{2}^{\epsilon}, \ae$ or $ b_{2}^{\epsilon} $.
\end{prop}

\medskip

A typical situation we encounter in the estimate of $ I_{j} $ is the
evaluation of a norm or scalar product involving a commutator. 
We have
\begin{prop}
\label{prop:comm}
The commutator 
\begin{equation}
\label{eq:commut}
[ a_{2}^{\epsilon}(t, x, D_{x}) , f^{-2N}] 
\end{equation}
has a symbol in $ S(f^{-2N} N \epsilon \langle \xi \rangle,
g^{\epsilon}) $. 
\end{prop}
\begin{cor}
\label{cor:comm}
If $ 0 < \epsilon \leq \epsilon_{0} $, where $ \epsilon_{0} $ denotes
a suitably small positive number depending on $ N $, then the
commutator in $(\ref{eq:commut})$ can be written as
\begin{equation}
\label{eq:comm2}
[ a_{2}^{\epsilon}(t, x, D_{x}) , f^{-2N}] = f^{-2N} \gamma_{1}^{\epsilon}(t, x, D_{x}), 
\end{equation}
where $ \gamma_{1}^{\epsilon} \in S(\langle \xi \rangle, g^{\epsilon}) $.
\end{cor}
\begin{proof}[Proof of Proposition $\ref{prop:comm}$]
Since $ f^{-2N} $ does not depend on $ x $, the
bracket can be written as a product: 
$$ 
\symb\left( \left [ a_{2}^{\epsilon}(t, x, D_{x}) , f^{-2N}(t, D_{x}) \right]
\right) = a_{2}^{\epsilon} f^{-2N} - f^{-2N} \# a_{2}^{\epsilon},
$$
where $ \symb(b) $ denotes the symbol of the operator $ b $.

Using Proposition \ref{prop:bony}, as well as definitions
(\ref{eq:10metric}) and (\ref{h}), we obtain that the r.h.s. of the
above identity belongs to $ S(f^{-2N} \epsilon \langle \xi\rangle,
g^{\epsilon}) $. 
\end{proof}
\begin{proof}[Proof of Corollary $\ref{cor:comm}$]
Choosing $ M=1 $ in (\ref{brem}), we get
\begin{multline*}
\sigma\left( \left [ a_{2}^{\epsilon}(t, x, D_{x}) , f^{-2N}(t, D_{x}) \right]
\right) (t, x, \xi) 
\\
= \int_{0}^{1} \frac{1}{(\pi \theta)^{2n}} \int
\int e^{- (2i/\theta) \sigma((x, \xi) - (z, \zeta), (x, \xi) - (y,
  \eta))} \frac{i}{2} D_{z}a_{2}^{\epsilon}(t, z, \zeta) 
\\
\cdot
D_{\eta} f^{-2N}(t, \eta) dz d\zeta dy d\eta d\theta.
\end{multline*}
Since
$$ 
\partial_{z} a_{2}^{\epsilon} \partial_{\eta} f^{-2N} 
= \frac{4N}{3}  \epsilon
  \langle(\partial_{z}a_{2})^{\epsilon}, \frac{\eta}{\langle \eta \rangle}\rangle f^{-2N}
  \langle \eta \rangle^{-1} \frac{\langle \eta \rangle^{-2/3}}{t + \langle \eta
    \rangle^{-2/3}},
$$
we see that besides the order function $ f^{-2N} \epsilon \langle
\xi\rangle $ we have also a factor $ N $, which justifies the presence
of $ \epsilon $. Here we used the notation $
(\partial_{z}a_{2})^{\epsilon}$ to denote the symbol $ (\partial_{z}
a_{2})(\epsilon^{2/3}t, \epsilon z, \xi) $. See also definition
(\ref{eq:epsdefa}). 
\end{proof}

Due to the above statements we may conclude that 
\begin{equation}
\label{eq:commutator}
[ f^{-2N} , a_{2}^{\epsilon} ] = f^{-2N} \alpha_{1}^{\epsilon},
\end{equation}
for some first order symbol $ \alpha_{1}^{\epsilon} $. Therefore
\begin{align}
\label{eq:I_{3}}
I_{3} &= \el \left( \langle f^{-2N} a_{2}^{\epsilon} u', u''\rangle +
  \langle u'' ,  \left( a_{2}^{\epsilon} f^{-2N} + f^{-2N}
    \alpha_{1}^{\epsilon} \right) u' \rangle \right)
\notag \\
&= \el  \left( \langle f^{-2N} a_{2}^{\epsilon} u', u''\rangle +
  \langle f^{-2N} a_{2}^{\epsilon} u'' , u' \rangle + \langle  f^{-2N}
  u'',  \alpha_{1}^{\epsilon}  u' \rangle \right)
\\
&= \el \partial_{t} \langle f^{-2N} a_{2}^{\epsilon} u' , u' \rangle +
\N \el \langle f^{-2N - 1} a_{2}^{\epsilon} u' , u' \rangle
\notag \\
&\phantom{=}
-\el \langle f^{-2N} \partial_{t}\left( a_{2}^{\epsilon}\right) u' , u'
\rangle
+ \el \langle  f^{-2N}
  u'',  \alpha_{1}^{\epsilon}  u' \rangle
\notag \\
&=
 \partial_{t} \langle \el f^{-2N} a_{2}^{\epsilon} u' , u' \rangle + 2\lambda \la \el \fN \aee u', u'\ra +
\N \el \langle f^{-2N - 1} a_{2}^{\epsilon} u' , u' \rangle
\notag \\
&\phantom{=}
+ I_{3,1} + I_{3,2}.
\notag
\end{align}
Here we denoted by $ u' = \partial_{t}u $ and $ u'' = \partial_{t}^{2}u $. Moreover
$ \partial_{t}\left( a_{2}^{\epsilon}\right)  $ denotes the operator
whose symbol (or coefficients in the differential case) are the $ t
$-derivative of $ a_{2}^{\epsilon} $. 

Repeating the same argument, we obtain
\begin{align} 
\label{eq:5.7}
\ep^{4/3} I_5 
&= -2 \theta \im \la \psi \fN \ae D_t u, \ae u \ra 
\notag \\
&= 
2 \theta \re \la
\psi \fN \ae u', \ae u \ra  \\ 
&= 
\theta [ \la \psi \fN \ae u', \ae u \ra + \la \ae u, \psi \fN \ae
u'\ra ]
\notag \\ 
&= 
\theta \pa_t  \la \psi \fN \ae u, \ae u \ra - \theta \psi' \la \fN
 \ae u, \ae u \ra + 2N/ 3 \theta \la \psi f^{-2N - 1} \ae u, \ae u \ra.
\notag
\end{align}

\subsection{Estimate of $I_4$}

Let us consider $ I_{4} $. We have
\begin{align} 
\label{eq:I4}
I_4 &= 
- 2 \im \psi \la \fN t \aee D_t u, \theta t \aee u \ra 
\notag \\
&= 2 \el \re
\la  \fN \aee u', \theta t \aee u \ra 
\notag \\ 
&= 
\theta \pa_t \Bigl(\el \la \fN \aee u , t \aee u \ra\Bigr) + 2\theta
t\lambda \el \la \fN \aee u,  \aee u \ra
\notag \\
&\phantom{=} 
+ \theta 2N/3 t \el \la  \fNt \aee u,  \aee u \ra - \theta \el \la  \fN
\aee u, \aee u \ra
\\
&\phantom{=}
- \theta \el \la \fN (\aee)_t u , t \aee u) - \theta \el \la \fN \aee
u, t (\aee)_t u \ra. 
\notag
\end{align}

Here we just used the fact that both $ f^{-2N} $ and $ a_{2}^{\epsilon} $ are
self adjoint in $ L^{2}(\Omega) $, $ t $ being a parameter at
this stage.

\subsection{Estimate of $I_2$ and $J_1$}

Let us consider the expression for $ I_{2} $, see (\ref{eq:4.8}),
\begin{align}
\label{eq:I_{2}}
I_{2} &= 2 \re \el \langle f^{-2N} \partial_{t}^{3} u, \theta 
a_{2}^{\epsilon} u \rangle \notag \\
&=
\theta \el \left( \langle f^{-2N} u''', a_{2}^{\epsilon} u \rangle +
  \langle a_{2}^{\epsilon} u, f^{-2N} u''' \rangle \right)
\notag \\
&=
\theta \el  \partial_{t} \bigg ( \langle f^{-2N} u'' , a_{2}^{\epsilon} u
\rangle + \langle a_{2}^{\epsilon} u , f^{-2N} u'' \rangle 
- \langle f^{-2N} u' , a_{2}^{\epsilon} u' \rangle \bigg) \\
&\phantom{=}
+ \N \theta \el \re\bigg ( \langle f^{-2N-1} u'' , a_{2}^{\epsilon} u
\rangle + \langle a_{2}^{\epsilon} u , f^{-2N-1} u'' \rangle 
\notag \\
&\phantom{=}
- \langle f^{-2N-1} u' , a_{2}^{\epsilon} u' \rangle \bigg)
+ \theta \el 2 \re \langle f^{-2N} \tilde{\alpha}_{1}^{\epsilon} u', u''
\rangle \notag \\
&\phantom{=}
 - \theta \el \re\bigg ( \langle f^{-2N} u'' , (\partial_{t}a_{2}^{\epsilon}) u
\rangle + \langle (\partial_{t} a_{2}^{\epsilon}) u , f^{-2N} u'' \rangle 
- \langle f^{-2N} u' , (\partial_{t} a_{2}^{\epsilon}) u' \rangle
\bigg).
\notag \\
&=
\theta  \partial_{t} \bigg ( 2 \el \re \langle f^{-2N} u'' , a_{2}^{\epsilon} u
\rangle  
- \el\re\langle f^{-2N} u' , a_{2}^{\epsilon} u' \rangle \bigg) 
\notag \\
&\phantom{=}
+ \N \theta \bigg ( 2 \el \re \langle f^{-2N-1} u'' , a_{2}^{\epsilon} u
\rangle 
- \el \re\langle f^{-2N-1} u' , a_{2}^{\epsilon} u' \rangle \bigg)
\notag \\
&\phantom{=}
+2 \lambda \theta \bigg ( 2 \el \re \langle f^{-2N} u'' , a_{2}^{\epsilon} u
\rangle 
- \el \re\langle f^{-2N} u' , a_{2}^{\epsilon} u' \rangle \bigg) + \sum_{k=1}^{4} I_{2,k}.
\notag
\end{align}
A few words are in order. Here $ \tilde{\alpha}_{1}^{\epsilon} $
denotes a suitable first order pseudodifferential operator originating
from a commutator exactly as it occurred for the other terms above. 

Moreover in deducing (\ref{eq:I_{2}}) the following identity has been
used:
\begin{multline*}
\partial_{t} \bigg(\langle f^{-2N} u'' , a_{2}^{\epsilon} u
\rangle + \langle a_{2}^{\epsilon} u , f^{-2N} u'' \rangle 
- \langle f^{-2N} u' , a_{2}^{\epsilon} u' \rangle \bigg) 
\\
=
\langle f^{-2N} u''', a_{2}^{\epsilon} u \rangle +
\langle f^{-2N} u'', a_{2}^{\epsilon} u' \rangle +
\langle a_{2}^{\epsilon} u' , f^{-2N} u'' \rangle
+ \langle a_{2}^{\epsilon} u, f^{-2N} u''' \rangle 
\\
- \langle f^{-2N} u'', a_{2}^{\epsilon} u' \rangle 
- \langle f^{-2N} u' , a_{2}^{\epsilon} u'' \rangle
+ \text{terms being } \mathscr{O}(N) 
+ \text{ terms involving } \partial_{t}a_{2}^{\epsilon}.
\end{multline*}
Thus let us examine the first four terms in the r.h.s. above. We have
\begin{multline*}
\langle f^{-2N} u''', \aee u \rangle 
 +\langle \aee u' , f^{-2N} u'' \rangle
+ \langle a_{2}^{\epsilon} u, f^{-2N} u''' \rangle 
- \langle f^{-2N} u' , \aee u'' \rangle
\\
=
\langle f^{-2N} u''', \aee u \rangle
+ \langle \aee u, f^{-2N} u''' \rangle 
+ \langle [ f^{-2N} , a_{2}^{\epsilon} ] u', u'' \rangle
\\
=
2 \re \langle f^{-2N} u''', \aee u \rangle
+ \langle f^{-2N} \tilde{\alpha}_{1}^{\epsilon}  u', u'' \rangle.
\end{multline*}
To obtain the last line we used Corollary \ref{cor:comm} and the fact
that $ a_{2}^{\epsilon} $ is a self-adjoint operator.\\

For $J_1$ we use the same argument and we obtain
\begin{align}
\label{eq:5.10}
\ep^{2/3} J_1 &= 2 \theta\im \la \psi \fN D_t^3 u, \ae u \ra 
\notag \\
&= 2 \theta \re \la
\psi \fN \pa_t^3 u, \ae u \ra
\notag \\ 
&= \theta \la \psi \fN \pa_t^3 u, \ae u \ra + \theta \la \ae u , \psi
\fN \pa_t^3 u \ra
\notag \\  
&= \theta \pa_t \Bigl( 2\psi \re \la \fN u'', \ae u\ra   - \re \la
\psi \fN u', \ae u'\ra\Bigl)
\notag \\
&\phantom{=}
- \theta\psi'\Bigl( 2 \re \la \fN u'', \ae u \ra - \re \la \fN u', \ae
u' \ra \Bigr)
\notag \\
&\phantom{=}
+ \theta 2N/3 \Bigl( 2\psi  \re\la f^{-2N - 1} u'', \ae u \ra - \re
\la \psi f^{-2N - 1} u', \ae u' \ra \Bigr)
\notag \\
&\phantom{=}  
+ \theta 2\psi  \re \la f^{-2N} \beta_1 u', u'' \ra
\end{align}
with a first order operator $\beta_1$.

\subsection{Estimate of $J_2$ and $J_4$}

We have
\begin{align} 
\label{4}
\ep^{2/3}(J_2 + J_4) &= - 2\im \la \psi \fN t \aee D_t u, \theta \ae u
\ra - 2 \im \la \psi \fN \ae D_t u, \theta t \aee u \ra 
\notag \\
&= 
2 \re \Bigl[\la \psi \fN t \aee u' , \theta \ae u \ra +  \la \psi \fN
\ae u', \theta t \aee u \ra\Bigr]
\notag \\
&= \theta 2 \pa_t \Bigl(\re \la \el \fN \aee u, \ae u \ra \Bigr)
 + \theta 4N/3\re \la \el \fNt \aee u, \ae u \ra 
\notag \\ 
&\phantom{=}
+4 \theta \lambda \re \la \el \fN \aee u, \ae u \ra  - 2 \theta \re
\la \el \fN (\aee)_t u, \ae u \ra. 
\end{align}

\subsection{Estimate of $J_3$}

Consider, see (\ref{eq:4.8}),
\begin{align*}
\ep^{2/3} J_3 &= 2\re \la \psi \fN \ae u', u'' \ra
\\
& = \psi \la  f^{-2N}\ae u',  u''\ra +  \psi \la u'', \ae f^{-2N} u'
\ra + \psi \re\la f^{-N} u'', \gamma_1 f^{-N} u'\ra 
\\
&= \pa_t \la \psi f^{-2N} \ae  u', u'\ra  - \psi'\re \la f^{-2N} \ae u', u'\ra  + J_{3.1}
\\
&\phantom{=}
 + 2N/3 \la \psi f^{-2N - 1}\ae  u',  u' \ra, 
\end{align*}
where
$$
J_{3,1} = \re\la \psi  f^{-N} u'', \gamma_1 f^{-N} u'\ra
$$
and $\gamma_1$ is a first order operator.

\subsection{Estimate of $ A_1$}

Let us rewrite $ i A_1 $ in the following way
\begin{align*}
i A_1 &= 2i \im \langle \el f^{-2N} \epsilon^{1/3} t b_{3}^{\epsilon}
u,  u'' \rangle \\
&=
\epsilon^{1/3} t \el  \bigg(  \langle f^{-2N}  b_{3}^{\epsilon} u, u''
\rangle - \langle u'', f^{-2N} a_{3}^{\epsilon} u \rangle \bigg).
\end{align*}
We have the identity
\begin{align}
\label{eq:ident7}
\partial_{t} 2i \im \langle f^{-2N}  b_{3}^{\epsilon} u, u' \rangle 
&= 
\langle f^{-2N}  b_{3}^{\epsilon} u, u'' \rangle 
- \langle u'', f^{-2N} b_{3}^{\epsilon} u \rangle 
+ \langle f^{-2N}  b_{3}^{\epsilon} u', u' \rangle
-\langle u' , f^{-2N} b_{3}^{\epsilon} u' \rangle \\
&\phantom{=}
- \N \bigg( \langle f^{-2N-1}  b_{3}^{\epsilon} u, u' \rangle 
- \langle u' , f^{-2N-1} b_{3}^{\epsilon} u \rangle \bigg) 
\notag \\
&\phantom{=}
+ \langle f^{-2N} (\partial_{t} b_{3}^{\epsilon} ) u , u' \rangle
- \langle u' , f^{-2N} (\partial_{t} b_{3}^{\epsilon} ) u \rangle.
\notag
\end{align}
Plugging (\ref{eq:ident7}) into the above expression for $ iA_1 $,
we then obtain
\begin{align}
\label{eq:I_{7}}
A_{1} &= \epsilon^{1/3}  t \el \partial_{t} \left( 2 \im \langle
  f^{-2N} b_{3}^{\epsilon} u, u' \rangle \right) 
+ 2\epsilon^{1/3} t \el \N  \im \langle
  f^{-2N-1} b_{3}^{\epsilon} u, u' \rangle \nonumber \\
&\phantom{=}
- \epsilon^{1/3} t \el 2 \im \langle f^{-2N} b_{3}^{\epsilon} u' , u'
\rangle  
-2\epsilon^{1/3} t \el \im \langle f^{-2N} (\partial_{t}
b_{3}^{\epsilon}) u , u' \rangle 
\nonumber \\
&= 2\epsilon^{1/3}  \partial_{t} \left( t \el  \im \langle
  f^{-2N} b_{3}^{\epsilon} u, u' \rangle \right) - 2\epsilon^{1/3} \el   \im \langle
  f^{-2N} b_{3}^{\epsilon} u, u' \rangle \nonumber \\
&\phantom{=}
+ 2\lambda \epsilon^{1/3}  t \el   \im \langle
  f^{-2N} b_{3}^{\epsilon} u, u' \rangle
+ \epsilon^{1/3} t \el 4N/3  \im \langle
  f^{-2N-1} b_{3}^{\epsilon} u, u' \rangle 
\nonumber \\
&\phantom{=}
+ \sum_{k=1}^{2} A_{1, k}.
\end{align}

\subsection{Estimate of $ A_2$ and $A_3$}
Using the calculus it is not difficult to show that there is a
symbol of first order,  $ \hat{\alpha}_{0}^{\epsilon} $, such that
\begin{equation}
\label{eq:5.14}
A_{2} = 2 \epsilon^{1/3}  \theta t^2 \el  \im \langle f^{-2N} 
 b_{3}^{\epsilon} u, a_{2}^{\epsilon}  u \rangle 
= \epsilon^{1/3}  \theta t^2 \el 2\im \langle f^{-2N}
\gamma_0^{\epsilon} a_{2}^{\epsilon} u,
a_{2}^{\epsilon} u \ra.
\end{equation}
Here our argument is based on the fact the principal symbols of the operators
$a_2^{\ep}(t, x, D_x)$ and $b_3^{\ep}(t, x, D_x)$ are real-valued.
 Thus we obtain
$$
(a_2^{\ep}(t, x, D_x) b_3^{\ep}(t, x, D_x))^* = b_2^{\ep}(t, x, D_x)
a_3^{\ep}(t, x, D_x)) + \alpha_4^{\epsilon}(t, x, D_x)
$$
with a pseudodifferential operator $\alpha_4^{\epsilon}$ of order
4. Since $a_2^{\ep}(t, x, D_x)$ is elliptic, it is easy to find a zero order operator
$\gamma_0^{\ep}(t, x, D_x)$ so that $\alpha_4^{\epsilon}(t, x,
D_x) = (a_2^{\ep})^*  \gamma_0^{\ep}(t, x, D_x) a_2^{\ep}.$ 

\bigskip

For $A_3$ we obtain straightforwardly that
$A_3 = 2 \theta \ep^{-2/3}\im\la \psi f^{-2N} \ep^{1/3}t^2 b_3^{\ep} u , \ae u\ra.$

\section{Energies}
%
%
\setcounter{equation}{0}
\setcounter{thm}{0}
\setcounter{prop}{0}  
\setcounter{lem}{0}
\setcounter{cor}{0} 
\setcounter{defn}{0}
\setcounter{rem}{0}
%

Summarizing the expression of all terms in Section 5, we may rewrite
(\ref{eq:4.8}) in the following form 
\begin{eqnarray}
\label{eq:6.1}
-2\im \langle f^{-2N} \mathcal{P} u, Mu \rangle = \partial_{t}
\mathscr{E}_N(u)  + \N \mathscr{E}_{N + 1/2}(u)
+ 2\lambda \mathscr{E}_{N}(u)
+ \mathscr{R},
\end{eqnarray}
where
\begin{align}
\label{eq:EN}
\mathscr{E}_N(u) & = \psi \Bigl[\| f^{-N} u''\|_0^{2}   
                  +(1 - \theta) \re \la f^{-2N} (t \aee + \ep^{-2/3} \ae) u', u' \ra 
\\
&\phantom{=}
    + \theta \|f^{-N} (t \aee u + \ep^{-2/3} \ae ) u\|_0^2 
+ \theta 2\re \la f^{-2N} u'', (t \aee + \ep^{-2/3} \ae) u\ra \Bigr]
\notag \\                                                  
&\phantom{=}
+ \epsilon^{1/3} t \el 2 \im \la f^{-2N} b_{3}^{\epsilon} u,
u'\ra .   
\notag             
\end{align}
Moreover $ \mathscr{R} $ includes all terms that can be considered
``errors'', since they do not contribute to the energy $
\mathscr{E}_{N} $ or $ \mathscr{E}_{N+1/2} $. Note that somewhat
improperly we include into $ \mathscr{R} $ also norms multiplied by $
\frac{\el}{t^2} $ which are positive, but require ``ad hoc'' treatement. Recall that we have $0 \leq t_0 \leq t \leq T$ and we suppose that $u(t_0, x) = u_t(t_0, x) = u_{tt}(t_0, x) = 0$, so the expression with factor $\psi$ or $\frac{\el}{t^2}$ below make sense when $t_0 = 0$.

The quantity $ \mathscr{R} $ is defined as
\begin{align}
\label{eq:RR}
\mathscr{R} &= \frac{\el}{t^2} \Bigl(\| f^{-N} u''\|_0^{2}  
+ \theta \ep^{-4/3} \la f^{-2N} \ae u, \ae u \ra  
+(1- \theta)  \ep^{-2/3}\re\la f^{-2N} \ae  u', u'\ra   
\\
& \phantom{=}
+ \theta\ep^{-2/3} 2\re \la \fN u'', \ae u \ra\Bigr)   
  -\theta \el \|f^{-N} \aee u\|_0^2     
\notag                        
\\& \phantom{=}
- 2 \theta \ep^{-2/3}\el \re \la f^{-2N} (\aee)_t u, \ae u \ra  
\notag
\\
&\phantom{=}
 - 2 \theta \el \re  \la f^{-2N}(\aee)_{t} u, t \aee u\ra             
+ e^{-2\lambda t} \left( -\la f^{-2N} (\aee)_t u', u' \ra 
+ \la f^{-2N} u'', \alpha_{1}^{\epsilon} u'\ra \right)   
\notag
\\
&\phantom{=}
+ \theta \el 2 \re \langle f^{-2N} \tilde{\alpha}_{1}^{\epsilon} u',
u'' \rangle  
+ 2 \theta \ep^{-2/3}\psi \re \la \beta_1  f^{-N} u', f^{-N} u''\ra   
+\ep^{-2/3}\psi \re  \la \gamma_1 f^{-N} u', f^{-N} u'' \ra  
\notag
\\
&\phantom{=}
 - \theta \el \bigg ( 2 \re \langle f^{-2N} u'' , (\aee)_t u
\rangle                           
- \re \langle f^{-2N} u' , (\aee)_t u' \rangle\bigg)    
\notag                        
\\                          
&\phantom{=}
- 2\epsilon^{1/3} \el   \im \langle  f^{-2N} b_{3}^{\epsilon} u, u'
\rangle                                          
- \epsilon^{1/3} t \el 2 \im \langle f^{-2N} b_{3}^{\epsilon} u' , u'
\rangle                                          
\notag
\\
&\phantom{=}
-2\epsilon^{1/3} t \el \im \langle f^{-2N} (\partial_{t}
b_{3}^{\epsilon}) u , u' \rangle             
\notag                       
\\
&\phantom{=}
+ \epsilon^{1/3}  \theta t^2 \el 2\im \langle f^{-2N}
\gamma_0^{\epsilon} a_{2}^{\epsilon} u,
a_{2}^{\epsilon} u \ra
\notag                      
\\
&\phantom{=}
+ 2 \theta \ep^{-1/3}\el t\im\la f^{-2N}  b_3^{\ep} u
, \ae u \ra 
\notag                      
\\
&\phantom{=}
+ \sum_{\nu= 1}^3 B_{\nu} + \text{lower order terms}.
\notag
\end{align}
To keep the exposition simple it is convenient to denote by $\sum_{j=
  1}^{10} {\mathscr R}_j$ the sum of 10 terms corresponding to the 10
lines in the expression of ${\mathscr R}$ above.

Consider  the sum
$$ 
S_{k} = \psi\|f^{-k} u''\|_0^2 + \theta \psi  \|f^{-k} (t\aee +
\ep^{-2/3} \ae) u \|_0^2 + 2\theta \psi \re \la f^{-k} (t\aee +
\ep^{-2/3} \ae) u ,f^{-k} u'' \ra. 
$$
Choose $\theta = 1/3$. Then we deduce
\begin{equation} 
\label{eq:6.4}
S_k = \frac{1}{3} \psi \|f^{-k} u''\|_0^2 + \frac{1}{6} \psi \|f^{-k}(t
\aee u + \ep^{-2/3} \ae)u \|_0^2 + \frac{2}{3}\psi \|f^{-k}(u''
+\frac{1}{2} (t \aee + \ep^{-2/3} \ae)u))\|_0^2. 
\end{equation}
It is clear that
$$
\|f^{-k}(u'' + \frac{1}{3} (t \aee + \ep^{-2/3} \ae)u)\|_0^2 \leq
2\|f^{-k}(u'' + \frac{1}{2} (t \aee + \ep^{-2/3} \ae)u)\|_0^2 +
\frac{1}{18} \|f^{-k}(t \aee + \ep^{-2/3} \ae)u\|_0^2
$$
and we get
\begin{equation} 
\label{eq:mul}
 S_k \geq \frac{1}{3}\psi \|f^{-k} u''\|_0^2 +
\frac{4}{27}\psi \|f^{-k} (t \aee + \ep^{-2/3} \ae) u)\|_0^2 +
\frac{1}{3} \psi\|f^{-k} (u'' +\frac{1}{3} (t \aee + \ep^{-2/3} \ae)
u)\|_0^2. 
\end{equation}
In the same way we obtain
\begin{eqnarray*}
{\mathscr R}_1 + {\mathscr R}_2 &=&  \fr{\el}{t^2} \Bigl(\fr{1}{3}
\big\|f^{-N} u''\|_0^2 + \fr{1}{6} \ep^{-4/3}\big\|f^{N} \ae u\|_0^2 +
\fr{2}{3} \|f^{-N} \Bigl(u'' + \fr{1}{2} \ep^{-2/3} \ae u\Bigr)\|_0^2 \\
&  & 
+ \fr{2}{3} \ep^{-2/3} \re \la f^{-2N} \ae u', \ae u'\ra - \fr{1}{3} \|f^{-N} t \aee u\|_0^2\Bigr) \\
&\geq  & 
\fr{\el}{3t^2}\Bigl( \|f^{-N} u''\|_0^2 + \fr{4}{9}
 \ep^{-4/3}\|f^{N} \ae u\|_0^2 +  \big\|f^{-N} \Bigl(u'' +  \ep^{-2/3}
 \ae u\Bigr)\big\|_0^2 \\
&  &
+ 2 \ep^{-2/3} \re \la f^{-2N} \ae u', \ae u'\ra -  \|f^{-N} t \aee u\|_0^2 \Bigr).
\end{eqnarray*}
To simplify the notations in the following we will write $\a0$ for
$a_2^{\ep}(0, x, D_x)$, while $a_2^{\ep}$ will denote the operator $a_2^{\ep}(t, x, D_x).$
\def\a0{{\bf a}_2^{\ep}}
Therefore $\aee(t, x, \xi) = \a0(0, x, \xi) + \ep^{2/3} t \tilde{a}_2^{\ep}(t, x, \xi)$ and we have
$$\frac{2}{3} \psi\re \la f^{-k} (ta_{2}^{\epsilon} + \ep^{-2/3} \ae) u' , f^{-k} u'\ra + \frac{1}{6} \psi\| f^{-k} (t\aee + \ep^{-2/3} \ae) u \|_0^{2}$$
$$=\frac{2}{3}\psi \re \langle f^{-k} (t \a0 + \ep^{-2/3} \ae) u' , f^{-k} u'\ra 
 + \frac{1}{6} \psi \| f^{-k}(t \a0 + \ep^{-2/3} \ae)u \|_0^{2} +t\el \ep^{2/3}A_{k}^{(2)}(u).$$
Here $t \ep^{2/3}A^{(2)}_{k}(u)$ denote a sum of terms which have
coefficient $t\ep^{2/3}$. It will be easy to absorb them taking $\ep$
and $t$ small and we will discuss this in Section 7 after the terms in
${\mathscr E}_N$ will have been conveniently prepared. To start we need the following
\begin{prop}
\label{prop:weighted-garding}
There exist positive constants $ C_{1} $ and $ C_{2} $, independent of
the positive integer $ k $ such that for $0 < \ep \leq \ep_0(k)$ we have
\begin{equation}
\label{eq:6.5}
\re \langle f^{-k} \a0 v, f^{-k} v
\rangle \geq C_{1} \| f^{-k} v\|_{1}^{2} - C_{2} \|
f^{-k} v \|_{0}^{2},
\end{equation}
for every $ v \in C_{0}^{\infty} $. The constant $C_1$ depends only on the symbol $a_2^{\ep}(0, x, \xi).$
\end{prop}
\begin{proof}
The proof consists in just making sure that we may commute the weight
operator $ f^{-k} $ with $ a_{2}^{\epsilon} $ and estimate the
errors, which naturally depend on $ N $. We have that
$$
\langle f^{-k} a_{2}^{\epsilon}(0, x, D_{x}) v, f^{-k} v
\rangle 
= \langle \a0 f^{-k} v, f^{-k} v
\rangle + \langle [f^{-k} , \a0] v, f^{-k} v
\rangle
= X_{1} + X_{2}.
$$
Keeping in mind that $ a_{2}^{\epsilon} $ is uniformly elliptic and using the
strict G\aa rding inequality for it, we obtain that
$$ 
X_{1} \geq c_{1} \| f^{-k} v \|^{2}_{1} - c_{2} \|
f^{-k} v \|_{0}^{2},
$$
for two suitable positive constants $ c_{1} $ and $ c_{2} $
independent of $ k $. We are
thus left with $ X_{2} $. By Proposition \ref{prop:comm} and Corollary
\ref{cor:comm} we see that if $ \epsilon $ is small enough depending
on $ k $, i.e. if $
\epsilon \leq \epsilon_{0}(k) $, there is a positive constant $ c_{3} $ independent of $k$,
such that 
$$ 
| X_{2} | \leq c_{3} \| f^{-k} v \|^{2}_{1/2} \leq \delta
\| f^{-k} v \|_{1}^{2} + c_{3}' \delta^{-1} \|
f^{-k} v \|^{2}_{0}.
$$
Choosing $ \delta $ conveniently small, but independent of $ \epsilon
$ and $ k $, we obtain the assertion.
\end{proof}
To treat the negative terms, we apply the
following lemma which will play a key role in the next section. 
\begin{lem}
\label{lemma:weight-inequality}
For $t \geq 0 $ and $ \xi \in \R^n $ we have
\begin{equation}
\label{eq:6.6}
\frac{1}{\langle\xi \rangle^{2}} +  t f^{2} \geq  f^{3}.
\end{equation}
\end{lem}
\begin{proof}
The proof is a simple verification. In fact $ f^{3} = f^{2} t/3 + f^{2}
\langle\xi\rangle^{-2/3} $. The latter quantity is equal to $ f^{2} t/3
+ \langle \xi\rangle^{-2} + \frac{2}{3} t \langle\xi\rangle^{-4/3} + \frac{t^{2}}{9}
\langle \xi \rangle^{-2/3} $. It is clear that 
$$
t \Bigl[\frac{2}{3} \la  \xi \ra^{-4/3} + \frac{t}{9} \la \xi
\ra^{-2/3}\Bigr] \leq \frac{2}{3} t f^2
$$
and this accomplishes the proof. 
\end{proof}
To examine the term $J_{\alpha, k}= \psi \ep^{-2/3} \re \la f^{-k} \ae u ,
f^{-k} u \ra$, we use the fact that the operator $\ae$ is positive and
write
$$
J_{\alpha, k} \geq \ep^{-2/3}\psi \re \la [f^{-k}, \ae] u, f^{-k}
u\ra.
$$
The symbol of the operator $\ae$ is $ \mathscr{O}(\epsilon^{2}) $
uniformly and for $0 < \ep < \ep_1(k)$ we obtain 
$$
|J_{\alpha, k}| \leq c_4 \ep^{1/3} \psi\|f^{-k} u\|_{1/2}^2
$$
with $c_4$ independent on $\ep$ and $k$. Now an application of  Lemma
6.1 yields
$$
\la \xi \ra f^{-2k} \leq t \la \xi \ra f^{-2k -1} + \la \xi \ra ^{-1}
f^{-2k-3} \leq \la \xi \ra^{-1/3} t \la \xi \ra ^2f^{-2k} + \la \xi
\ra ^{-1/3} f^{-2k - 2}
$$
since $\la \xi \ra^{-2/3} f^{-1} \leq 1$. Therefore
$$
|J_{\alpha, k}| \leq c_4 \ep^{1/3} \el \|f^{-k}u\|_{1/2}^2 + c_4 \ep^{1/3}
\psi \|f^{-k - 1} u\|_0^2
$$
and for small $0 < \ep \leq \ep_1(k)$ taking into account
(\ref{eq:6.5}), we get
\begin {eqnarray} 
\label{eq:6.7}
2 \psi |\re \la f^{-k} ( t \a0 + \ep^{-2/3} \ae) u', f^{-k} u'\ra |&
\geq &  C_3 \el \|f^{-k} u'\|_1^2 \nonumber \\
- C_2 \el \|f^{-k}u' \|_0^2
 - c_4 \ep^{1/3}\psi \|f^{-k - 1} u' \|_0^2.
\end{eqnarray}
\def\a0{{\bf a}_2^{\ep}}
\def\b0{{\bf b}_2^{\ep}}
\def\be{b_3^{\ep}}

We introduce now the energy by the following
\begin{defn}
For a non negative integer $ k $ we define the $ k $-th energy as
\begin{align*}
E_{k}(u) & = \psi\Bigl(\frac{1}{3} \| f^{-k} u'' \|_0^{2} +\frac{2}{3}
\re \la f^{-k}(t \a0 + \ep^{-2/3} \ae) u' , f^{-k}u' \ra
\\
&\phantom{=}  
+ \frac{1}{6} \| f^{-k}(t \a0 u + \ep^{-2/3} \ae) u \|_0^{2} +
\frac{2}{3} \|f^{-k}( u'' +\frac{1}{2} (t \a0 + \ep^{-2/3}
\ae)u)\|_0^2\Bigr). 
\end{align*}
\end{defn}
For the expression of the energy  ${\mathscr E}_k(u)$, $ k = N$, $ N+
1/2$ we have, with the notations above, the representation
$$
{\mathscr E}_k(u) = E_k(u) + t \ep^{2/3}\el A_k^{(2)}(u) + 2
\ep^{1/3}t \el \im \la f^{-k} \be u,f^{-k} u' \ra.
$$
The last two terms on the right hand side can be considered as small
perturbations. 

Moreover in the energy $E_{N+ 1/2}(u)$ we have no positive terms
involving 
$$
\|f^{-N - 3/2}u\|_1^2, \quad  \|f^{-N - 3/2} u'\|_0^2 \quad \text{ and
} \quad  \|f^{-N - 5/2} u\|_0^2.
$$ 
These turn out very useful in order to absorb a number of ``errors''
using Lemma 6.1.

For this purpose we will obtain several new positive terms exploiting
a part of the energy $E_{N+ 1/2}(u).$ 
The same argument applies to the energy $E_N(u).$ 

Let us now consider the following identity, where $ k $ is a positive
integer and $g$ denotes a smooth function in the same class as $ u
$:
$$ 
\psi f^{-2k} 2 \re g'\bar{g} = \pa_t \Bigl( \psi f^{-2k} |g|^2 \Bigr)
- \psi'f^{-2k} |g|^2 + 2k/3  \psi f^{-2k - 1} |g|^2
$$
which implies
$$ 
\psi f^{-2k + 1} |g'|^2 \geq \pa_t \Bigl(\psi  f^{-2k} |g|^2 \Bigr)  -
\psi'f^{-2k} |g|^2 + (2k/3 -1) \psi  f^{-2k - 1} |g|^2.
$$
Taking $g = \pa_t u$, $k = N+1$, we have
\begin{equation} 
\label{eq:6.8}
\psi f^{-2N- 1} |u''|^2 \geq \pa_t \Bigl( \psi f^{-2N - 2} |u'|^2 \Bigr) - \psi'f^{-2N - 2} |u'|^2 + (2N/3 -1/3) \psi f^{-2N - 3} |u'|^2,
\end{equation}
while taking $g = u$, $k = N +2 $, we get
\begin{equation} 
\label{eq:6.9}
\psi f^{-2N - 3} |u'|^2 \geq \pa_t \Bigl( \psi  f^{-2N - 4} |u|^2
\Bigr) - \psi' f^{-2N - 4} |u|^2 + (2N/3 +1/3) \psi  f^{-2k - 5}
|u|^2. 
\end{equation}
Combining (\ref{eq:6.8}) and (\ref{eq:6.9}), we get
\begin{align} 
\label{eq:6.10}
\psi \|\fNt u'' \|^2 
&\geq 
\pa_t \Bigl( \psi \|f^{N -1}  u'\|^2 \Bigr) - \psi'\|f^{-N -1} u'\|_0^2
\notag \\
&\phantom{\geq} 
 + (2N/3 - 4/3) \psi \|f^{-N - 3/2} u'\|_0^2 + \pa_t \Bigl( \psi \|f^{-N
   - 2} u \|_0^2\Bigr) 
\notag \\
&\phantom{\geq}
- \psi'\|f^{-N - 2} u \|_0^2 + \frac{2N + 1}{3} \psi \|f^{N - 5/2} u\|_0^2.
\end{align}
Also we obtain easily the inequality
\begin{eqnarray} 
\label{eq:6.11}
e^{-2\lambda t} \| f^{-N-\frac{1}{2}} \partial_{t}u \|^{2}_{1} & \geq & \partial_{t}
\left( e^{-2\lambda t} \| f^{-N-1} u \|^{2}_{1} \right) +
2\lambda e^{-2\lambda t} \| f^{-N-1} u \|_{1}^{2} 
\nonumber \\
&   &
+ \frac{1}{3}(2N+1) e^{-2\lambda t} \| f^{- N - \frac{3}{2}} u \|^{2}_{1}.
\end{eqnarray}
By using the calculus of pseudodifferential operators, Proposition 6.1
and (\ref{eq:6.7}), we may write
\begin{multline} 
\label{eq:6.12}
2 \psi \re \la f^{-N - 1/2} (t \a0 + \ep^{-2/3} \ae) u', f^{-N - 1/2} u'\ra 
\geq  2C_3 \el\|\fNt u'\|_1^2 
\\
- C_4 \el\|f^{-N- 1/2} u'\|_0^2 - C_5 \ep^{1/3} \psi\|f^{-N - 3/2} u'\|_0^2.
\end{multline}
Next, taking into account the inequalities 
(\ref{eq:6.10}), (\ref{eq:6.11}) and (\ref{eq:6.12}), for small $\ep$ we get 
\begin{multline} 
\label{eq:6.13}
\frac{1}{3} \psi \| f^{-N - 1/2} u''\|_0^{2} + \frac{2}{3} \psi \re \la f^{-N - 1/2}( t\a0 + \ep^{-2/3} \ae) u',
  f^{-N-1/2} u' \ra  \\
 \geq   \frac{1}{3} \Bigl[\pa_{t} \left( \psi \| f^{-N-1}  u' \|_0^{2}
\right) -\psi' \| f^{-N-1} u' \|_0^{2}  \\
 + (2N-5)/3 \psi  \| f^{-N-\frac{3}{2}}  u' \|_0^{2}  
+ \partial_{t} \left( \psi \| f^{-N-2} u \|_0^{2} \right) -\psi' \| f^{-N-2} u \|_0^{2}   \\
+ (2N+ 1)/3 \psi \| f^{-N-\frac{5}{2}} u \|_0^{2}\Bigr]\\
+ \frac{1}{3} C_3 \Bigl[ \partial_{t}
\left(  \el \| f^{-N-1} u \|^{2}_{1} \right) +
 2\lambda \el \| f^{-N-1} u \|_{1}^{2} 
+ (2N + 1)/3 \el \| f^{- N - \frac{3}{2}} u \|^{2}_{1}\Bigr]  \\
+ \frac{1}{3} C_3  \el\|f^{-N- 1/2} u'\|_1^2- \frac{1}{3} C_4  \el\|f^{-N- 1/2} u'\|_0^2.
\end{multline}
Here the term $-C_5 \ep^{1/3} \psi \|f^{-N - 3/2} u'\|_0^2$ has been
absorbed by diminishing the coefficient in the term $\frac{2N-5}{3} \psi\|f^{-N
  -3/2} u'\|_0^2$ (compare with $\frac{2N - 4}{3} \psi \|f^{-N - 3/2} u'\|_0^2$ in the above
inequality.)

Therefore, using (\ref{eq:6.13}), we have for small $t$ and large $\lambda$ the estimate
\begin{multline} 
\label{eq:6.14}
 \frac{5N}{9} \psi \Bigl[\frac{1}{3} \|f^{-N-1/2}u'\|_0^2 + \frac{2}{3} \re \la f^{-N-1/2} (t \a0 + \ep^{-2/3} \ae)u' , f^{-N-1/2} u'\ra\Bigr]\\
+  \frac{N}{9}  \psi \left( \frac{1}{3} \| f^{-N - 1/2} u''\|_0^{2} + \frac{2}{3} \re \la f^{-N - 1/2}(t \a0 + \ep^{-2/3} \ae )u',
  f^{-N-1/2} u' \ra\right)  \\
 \geq \partial_t \Bigl( \frac{N}{27} \psi \| f^{-N-1}  u' \|_0^{2} + \frac{N}{27}  \psi \| f^{-N-2} u \|_0^{2}  + \frac{N}{27} C_3  \el  \| f^{-N-1} u \|^{2}_{1} \Bigr)
 \\
+ \frac{5N}{9}  \Bigl[\frac{1}{3} \psi\|f^{-N-1/2}u''\|_0^2 + \frac{2}{3} \el\re \la f^{-N-1/2}( t \a0 + \ep^{-2/3} \ae) u' , f^{-N-1/2} u'\ra\Bigr]
\\\
+\frac{N}{9} \Bigl[-\frac{\psi'}{3} \|f^{-N - 1} u'\|_0^2 + \frac{2N -5}{9} \psi\|f^{-N - 3/2} u'\|_0^2  - \frac{\psi'}{3}\|f^{-N - 2} u\|_0^2 \\
+ \frac{2N + 1}{9}\psi  \|f^{-N-5/2} u\|_0^2 + \frac{2C_3}{3} \lambda \el \|f^{-N - 1} u\|_1^2\\
+ \frac{(2N + 1)C_3}{9}\el\|f^{-N - 3/2} u\|_1^2 + \frac{ C_3}{3}  \el \|f^{-N - 1/2} u'\|_1^2\Bigr].
\end{multline}
Going back to the operator ${\mathcal P}$, we have
\begin{multline*}
-2 \im \la  \fN {\mathcal P} u, Mu \ra = 2 \psi \im  \la \fN {\mathcal
  P}u, (u'' + \fr{1}{3} \ep^{-2/3} \ae u) \ra + \frac{2}{3}\el \im
\la \fN {\mathcal P}u, \aee u\ra 
\\[10pt]
 \leq 7\el\|f^{-N} {\mathcal P} u\|_0^2  + \frac{\el}{6t^2}
\Bigl(\|f^{-N} (u'' + \fr{1}{3} \ep^{-2/3} \ae u)\|_0^2 + \fr{2}{3}\|
f^{-N} t \aee u\|_0^2\Bigr).
\end{multline*}
Therefore, exploiting (6.6), we obtain
\begin{multline} 
\label{eq:6.15} 
7\el\|f^{-N} {\mathcal P} u\|^2 \geq \pa_t {\mathscr E}_N(u)  + \frac{2N}{3} {\mathscr E}_{N+ 1/2}(u)  + 2\lambda {\mathscr E}_N (u) 
+ {\mathscr R} 
\\[10pt]
- \fr{\el}{6 t^2}\Bigl[\fr{1}{3}\|f^{-N}( u'' + \fr{1}{3} \ep^{-2/3}
\ae u)\|_0^2+ \frac{2}{3}\el\|f^{-N} \aee u\|_0^2 \Bigr]
\\[10pt]
= \pa_t {\mathscr E}_N(u)  + \frac{2N}{3} {\mathscr E}_{N+ 1/2}(u)  +
2\lambda {\mathscr E}_N(u)  +  {\mathscr Q}_1 + \sum_{j =
  3}^{10}{\mathscr R}_j + \sum_{j = \nu}^3 B_{\nu} +  \text{lower
  order terms}. 
\end{multline}
Here
\begin{multline*}
{\mathscr Q}_1 = \frac{\el}{ t^2} \Bigl(\frac{1}{3} \| f^{-N}
u''\|_0^{2}  +  \ep^{-4/3} \frac{4}{27} \|f^{-N} \ae u\|_0^2 +
\fr{5}{18}\| f^{-N} (u'' + \fr{1}{3}\ep^{-2/3} \ae u)\|_0^2
\\[10pt]
+ \fr{2}{3} \ep^{-2/3}\la f^{-N} \ae  u', f^{-N} u'\ra - \fr{4}{9}
\|f^{-N} t \aee u\|_0^2\Bigr).
\end{multline*}
Finally, taking into account (\ref{eq:6.4}), (\ref{eq:6.14}),
(\ref{eq:6.15}), we obtain the energy estimate 
\begin{align}
\label{eq:final}
7 \el \| f^{-N} \mathcal{P}u \|_0^{2} 
&\geq 
\partial_{t} \bigg( \mathscr{E}_{N} (u) +  \frac{N}{27} \psi \|
f^{-N-1}  u' \|_0^{2} + \frac{N}{27}  \psi \| f^{-N-2} u
\|_0^{2}\notag \\ 
&\phantom{\geq}  
+ \frac{N}{27} C_3  \el  \| f^{-N-1} u \|^{2}_{1} \bigg)\notag \\
&\phantom{\geq}
+ 2\lambda {\mathscr E}_{N}(u) + \ep^{1/3}t \el \frac{4N}{3} \im \la
f^{2N -1} b_3 u, u'\ra + \ep^{2/3} t A_{N + 1/2}^{(2)}(u)\notag \\
&\phantom{\geq}
+ N \psi \Bigl[\frac{5}{27} \|f^{-N-1/2}u''\|_0^2 + \frac{10}{27} \re
\la f^{-N - 1/2} ( t\a0 + \ep^{-2/3} \ae) u' , f^{-N-1/2} u'\ra \notag
\\ 
&\phantom{\geq}
+\frac{4}{9}  \|f^{-N - 1/2}(u'' + \frac{1}{2} (t \a0 + \ep^{-2/3} \ae)u)\|_0^2 + \frac{1}{9} \|f^{-N - 1/2} (t \a0 + \ep^{-2/3} \ae) u\|_0^2  \Bigr]
\notag \\
&\phantom{\geq}
+\frac{N}{9} \Bigl[-\frac{\psi'}{3} \|f^{-N - 1} u'\|_0^2 + \frac{2N
  -5}{9} \psi\|f^{-N - 3/2} u'\|_0^2  - \frac{\psi'}{3}\|f^{-N - 2}
u\|_0^2 \notag \\
&\phantom{\geq}
+ \frac{2N + 1}{9}\psi  \|f^{-N-5/2} u\|_0^2 + \frac{2C_3}{3} \lambda
\el \|f^{-N - 1}u\|_1^2 \notag \\
&\phantom{\geq}
+ \frac{(2N + 1)C_3}{9}\el\|f^{-N - 3/2} u\|_1^2 
+ \frac{ C_3}{3}  \el \|f^{-N - 1/2} u'\|_1^2\Bigr]
\notag \\
&\phantom{\geq}
+ {\mathcal Q}_1 +  \sum_{j=3}^{10} {\mathscr R}_{j} + \sum_{\nu = 1}^3 B_{\nu}
+ \text{lower order terms}.  
\end{align}

\section{Estimate of the error terms in the energy inequality}
%
%
\setcounter{equation}{0}
\setcounter{thm}{0}
\setcounter{prop}{0}  
\setcounter{lem}{0}
\setcounter{cor}{0} 
\setcounter{defn}{0}
\setcounter{rem}{0}
%

The last line of (\ref{eq:final}) contains a number of terms grouping the ``errors'' 
that must be dominated with the other positive terms.  We point out
that  the $ B_{\nu} $ are
just those ``errors'' associated with the lower order terms containing pure second
order $ x $-derivatives and did not play any
role up to now.
It is convenient to write 
$a_2^{\ep}(t, x, \xi) = a_2^{\ep} (0, x, \xi) + t\ep^{2/3}
\tilde{a}_2^{\ep} (t, x, \xi)$ and to replace the operator
$a_2^{\ep}(t, x, D_x)$ by the operator $\a0$ with symbol $a_2^{\ep}(0,
x, \xi).$ This will add a few terms similar to $t \ep^{2/3}
A^{(2)}_{N}(u)$ . Consequently, we have to deal with lower order terms which can be
treated choosing $\ep$ small,.

\bigskip

For the analysis of the terms $B_{\nu},\: \nu = 1,2, 3$ we apply a similar
procedure. First we write $b_2^{\ep}(t, x, \xi) = b_2^{\ep}(0, x, \xi)
+ t \ep^{2/3} \tilde{b}_2^{\ep} (t, x, \xi).$ The terms with the factor
$t\ep^{2/3}$ are similar to $ t \ep^{2/3} A^{(2)}_N(u)$ and can be
treated choosing $\ep$ small. We will analyze these terms in subsection 7.6. To keep the notation simple, we denote by $\b0$
the operator with symbol $b_2^{\ep}(0, x, \xi).$ The modified terms
$B_{\nu}$ will be denoted by $\tilde{B}_{\nu}, \: \nu = 1, 2, 3.$ 

\subsection{Estimate of $\|\b0 f^{-N + 1/2} u\|_0$}

 The subprincipal symbol of the operator ${\mathcal P}$ for $t= 0$ and $\tau = 0$ has the form
$$p_2'(0, x_0, \xi_0) = - \frac{i}{2} a_2^{\ep}(0, x_0, \xi_0) + b_2^{\ep}(0, x_0, \xi_0).$$
If we have a triple point $\rho = (0, x_0, \xi)$ for the  symbol $p_3(t, x, \tau, \xi)$, then $t = \tau = 0.$ 
Thus
$$ b_2^{\ep}(0, x, \xi) = \Bigl[ \frac{1}{2} i + \frac{p_2'(0, x, \xi)}{a_2^{\ep} (0, x, \xi)}\Bigr]a_2^{\ep}(0, x, \xi).$$
Let us introduce the number 
$$
\Pi = \frac{2}{3} + \max_{x \in \bar{U}_{x_0},\: \alpha(x, \xi) = 0, \: |\xi| = 1}
\Bigl|\frac{p_2'(0, x, \xi)}{a_2^{\ep}(0, x, \xi)}\Bigr|.
$$
which correspond to (1.7).
Here $U_{x_0} \subset \R^n$ is the open set defined in the hypothesis $(H_2)$.  Notice that we could have only one point $y \in U_{x_0}$ such that $\alpha(y, \xi) = 0.$
It is clear that for $V_{x_0} \Subset U_{x_0}$ sufficiently small we have
$$\sup_{x \in \bar{V}_{x_0},\: |\xi| = 1} \Bigl| \frac{1}{2}i  + \frac{p_2'}{a_2^{\ep}}(0, x,
\xi)\Bigr| \leq \Pi.$$
In the following we can assume that $u(t, x)$ has a  support with respect to $x$ included in $V_{x_0}$.
Let $\chi \in C^{\infty} (\R^n)$ be a function such that $0 \leq \chi(x) \leq 1,\:\chi(x) = 0$ for $|x| \leq 1$ and $\chi(x) = 1$ for $|x| \geq 2.$ We write
$$a_2^{\ep}(0, x, \xi)= \chi(\xi \delta)a_2^{\ep}(0, x, \xi) + (1- \chi(\xi \delta)) a_2^{\ep}(0, x, \xi)$$
with $\delta > 0.$
The operator with symbol $(1- \chi(\xi \delta))a_2^{\ep}(0, x, \xi)$
is smoothing and the analysis of the corresponding term is covered by
using the argument for lower order terms. On the other hand, the norm
in ${\mathcal L}(L^2(V_{x_0}))$ of the zero order operator
\begin{equation}
\label{eq:subcut}
\Bigl[ \frac{1}{2}i  + \frac{p_2'}{a_2^{\ep}}(0, x,
D_x)\Bigr]\chi(D_x \delta)
\end{equation}
is not greater than $\Pi$ if $\delta$ is chosen small enough depending
on the symbols $a_2$ and $p_2'$ (see Theorem 18.1.15 in
H\"ormander, \cite{h3}.) \\

Thus we have
$$\| \b0 f^{-N+ 1/2}u \|_0 \leq \big \|\Bigl [\frac{1}{2}i + \frac{p_2'}{a_2^{\ep}}(0, x, D_x)\Bigr] a_2^{\ep}(0, x, D_x) f^{-N + 1/2} u\big \|_0 + \|R_0\|_0$$
$$\leq \Pi \|\a0 f^{-N + 1/2} u\|_0 + \|R_0\|_0,$$
where $R_0$ is a lower order term including $\alpha_1 f^{-N + 1/2}$
with some first order pseudodifferential operator $\alpha_1$.

\bigskip

We are going to study the term
\begin{equation}
\label{eq:-N+undemi}
- \|\a0 f^{-N + 1/2} u\|_0^2 = -
 \Bigl[\re \la (a_2^{\ep})^2(0, x, D_x)
f^{-N + 1/2}u, f^{-N + 1/2}u\ra  +  |F_1|\Bigr],
\end{equation}
where $|F_1| \leq C_6 \|f^{-N +1/2}u\|_{3/2}^2$ and $ (a_2^{\ep})^2(0,
x, D_x) $ means $ \op\left( (a_{2}^{\epsilon})^{2}\right) $. 
We have the inequality
\begin{equation} \label{eq:7.3}
(a_2^{\ep})^2 \leq t (a_2^{\ep})^2 f^{-1} + a_2^{\ep} \alpha_0
f^{-3} \leq t^2 f^{-1} (a_2^{\ep})^2 f^{-1} + 2 t
f^{-2}a_2^{\ep}\alpha_{0}f^{-2} + f^{-3}\alpha_0^2 f^{-3},
\end{equation}
where $\alpha_0(x, \xi) = \frac{a_2^{\ep}(0, x, \xi)}{\la \xi \ra^2}.$
We can apply the sharp G\aa rding inequality and we estimate
$$
-\re \la (a_2^{\ep})^2 f^{-N + 1/2}u, f^{-N + 1/2} u \ra - |F_1| \geq
- t^2 \re \la(a_2^{\ep})^2 f^{-N - 1/2}u, f^{-N - 1/2}u\ra
$$
$$
- 2 t \re \la \alpha_{0} \a0 f^{-N - 3/2}u, f^{-N -3/2} u\ra - A_3^2 \|f^{-N -
  5/2} u\|_0^2 - Y_1 \|f^{-N + 1/2}u\|_{3/2}^2
$$
$$ 
\geq -t^2 \|\a0f^{-N - 1/2}u\|_0^2 - 2 t \re \la \alpha_{0} \a0 f^{-N - 3/2}u,
f^{-N -3/2} u\ra - A_3^2 \|f^{-N - 5/2} u\|_0^2
$$
$$
-Y_2 \|f^{-N +1/2}u\|_{3/2}^2 = \sum_{j=1}^4\Gamma_j.
$$
Here we have used the fact that
$$t^2 \re \la (a_2^{\ep})^2 f^{-N-1/2}u, f^{-N-1/2}u\ra =t^2 \|\a0 f^{-N-1/2}u\|_0^2 + F_2,$$
where $|F_2| \leq A_4 t^2 \|f^{-N- 1/2}u\|_{3/2}^2.$ Since $t^2 f^{-2}
\leq 9$, we included the term $F_2$ in the above sum taking $Y_2 \geq
Y_1.$ Notice that $Y_1, Y_2$ depend only on the symbol $a_2^{\ep}(0,
x, \xi).$

\bigskip

It is convenient to transform the term $\Gamma_4$. For this purpose we
use the inequality
$$
\la \xi \ra^3 f^{-2N + 1} \leq t \la \xi \ra ^3 f^{- 2N} + \la \xi
\ra f^{-2N - 2}\leq t^2 \la \xi \ra ^3 f^{-2N - 1} + t \la \xi \ra
f^{-2N - 3} + \la \xi \ra^2 f^{-2N - 2}
$$
$$
\leq \delta_1 t^2 \la \xi \ra^4 f^{-2N -1}+ D_{\delta_1} t^2 f^{-2N
  - 1} + 2t \la \xi \ra ^2 f^{-2N - 3} +   f^{-2N - 5}.
$$ 
Thus
$$\|f^{-N + 1/2} u\|_{3/2}^2 \leq \delta_1 t^2 \|f^{-N-1/2}u\|_2^2 + 2t \|f^{-N - 3/2}u\|_1^2 + \|f^{-N - 5/2}u\|_0^2 + D_{\delta_1} t^2 \|f^{-N - 1/2}u\|_0^2$$
$$\leq \delta_1 C^2 t^2\|\a0 f^{-N - 1/2} u\|_0^2 +2 t \|f^{-N - 3/2}u\|_1^2 + \|f^{-N - 5/2}u\|_0^2 + D'_{\delta_1} t^2 \|f^{-N - 1/2}u\|_0^2.$$
We take $\delta_1 > 0$ small enough so that $\delta_1 C^2 Y_2 \leq
\delta $ and we couple the term with the factor $t^2$  with
that also involving $t^2$ in $\Gamma_1$.  Next we fix $\delta_1$ and for small $t$ we have
$D'_{\delta_1} Y_2 t^2 \|f^{-N -1/2} u\|_0^2 \leq \|f^{-N - 5/2}
u\|_0^2$. Notice that we can choose $\delta > 0$ as small as we wish. We sum this term with $\Gamma_3$.  Consequently, we get 
\begin{eqnarray} 
\label{eq:7.4}
\sum_{j=1}^4 \Gamma_4 & \geq & - (1 + \delta) t^2 \|\a0 f^{-N-
  1/2}u\|_0^2-  2t (C_{a_2}+ Y_2 )\|f^{-N - 3/2}u\|_1^2\nonumber \\
 & &- (A_3^2+ 2Y_2)\|f^{-N - 5/2} u\|_0^2.
\end{eqnarray}
Here we have used that
$$\re \la \a0 f^{-N -3/2}u, f^{-N- 3/2}u\ra \leq C_{a_2} \|f^{-N- 3/2}u\|_1^2$$
and the constant $C_{a_2}$ depends on $a_2^{\ep}(0, x, \xi),$ while
$A_3 = \|\alpha_0(x, D_x)\|_{L^2(U) \to L^2(U)}$.

\bigskip
Summarizing we obtain the following
\begin{lem} 
Let $0\leq t_0 \leq t \leq T$ and let $D_t^k u(t_0, x) = 0,\:x \in V_{x_0},\: k = 0, 1, 2.$ For every fixed
small number $\delta > 0$ there exist constants $C_{a_2},\: Y_2,\:
A_3$ such that modulo lower order term $\psi R_0$  we have 
\begin{multline} 
\label{eq:7.5}
\psi \|\a0 f^{-N + 1/2} u \|_0^2 \leq  (1 + \delta) t\el \|\a0 f^{-N-
  1/2}u\|_0^2 +  2 \el (C_{a_2}+ Y_2 )\|f^{-N - 3/2}u\|_1^2  \\
 + (A_3^2+ 2Y_2)\psi\|f^{-N - 5/2} u\|_0^2.
\end{multline}
\end{lem}

We turn to the analysis of the term
\begin{align*}
K_{N+ 1/2} &=  \ep^{-2/3}\el 2\re \la f^{-2N-1} \aee u, \ae u \ra \\ 
&= 
\ep^{-2/3}\el 2\Bigl[ \re \la \aee(f^{-N - 1/2} u), \ae(f^{-N - 1/2}
u)\ra +   \re \la \alpha_1 \fNt u, \fNt \ae u\ra  \\
&\phantom{=}
+  \re \la \aee \fNt u, \beta_1 \fNt u \ra \Bigr]
\end{align*}
for some first order operators $\alpha_1, \beta_1$. First, notice
that for the principal symbol of $\ae$ we have  $\ae = \ep^2 \aez$
with a second order 
non-negative symbol $\aez$ and that the operator $\aez(x, D_x)$ is
self-adjoint. 
We replace $\ae$ by $\ep^2 \aez$ in the above equality and we
obtain a small factor $ \ep^{1/3}$. In fact for the two terms on the
right hand side we have factor $\ep^{4/3}$, while for the last one we
have $\ep^{1/3}$ related to the commutator $[\ae, f^{-N - 1/2}].$  For
the term $\re \la \aez \aee \fNt u, \fNt u \ra$ we 
can apply the sharp G\aa rding inequality since the principal symbol
of $\aez \aee$ is non-negative. The other terms in 
$K_{N+1/2}$ involve third order operators and we get
$$
K_{N + 1/2} \geq - C \ep^{1/3} \el\|\fNt u \|_{3/2}^2.
$$
An application of Lemma 6.1 yields
$$
\la \xi \ra ^3 f^{-2N - 1} \leq t \la \xi \ra^{3} f^{-2N - 2} + \la
\xi \ra f^{-2N - 4}  
\leq c_1 \la \xi\ra^{-1/3} \Bigl[t \la \xi \ra^4 f^{-2N - 1} + \la \xi
\ra ^2 f^{-2N - 3}\Bigr]
$$
since $\la \xi \ra ^{-2/3} f^{-1} \leq c.$ Thus
$$
K_{N + 1/2} \geq - C_1 \ep^{1/3}\el\Bigl(t \|\fNt u\|_2^2 + \|f^{-N -
  3/2} u\|_1^2\Bigr).
$$
On the other hand,
$$
\el\|f^{-N - 1/2 } u\|_2^2 \leq C_5 \el \|\a0 \fNt u\|_0^2 + C_6 \el
\|\fNt u\|_1^2
$$
and we deduce
$$
K_{N + 1/2} \geq - C_7 \ep^{1/3} \el \Bigl[t \|\a0 \fNt u \|_0^2 +
t\|\fNt u \|_1^2 + \|f^{-N - 3/2} u \|_1^2\Bigr].
$$
Combining this with (\ref{eq:7.5}), we deduce for small $\ep$ the estimate
\begin{eqnarray} 
\label{eq:7.6}
\psi \|\a0 f^{-N + 1/2} u \|_0^2 &\leq & (1 + \delta) \psi \|(t \a0  + \ep^{-2/3} \ae)f^{-N-
  1/2}u\|_0^2\nonumber \\
&  &
+  2 \el (C_{a_2}+ Y_2+ C_8 \ep^{1/3} )\|f^{-N - 3/2}u\|_1^2 
\nonumber \\
&  &
 + (A_3^2+ 2Y_2)\psi\|f^{-N - 5/2} u\|_0^2.
\end{eqnarray}

\subsection{Estimate of the sum $\sum_{\nu = 1}^3 \tilde{B}_{\nu}$}

We have
\begin{align*}
\sum_{\nu = 1}^3 \tilde{B}_{\nu} &= 2\im \psi \la f^{-N} \b0 u, f^{-N}
\Bigl(u'' + \frac{1}{3} (t \a0 + \ep^{-2/3} \ae)u\Bigr)\ra
\\
&= 
2\im \psi \la f^{-N} \b0 u, f^{-N} \Bigl(u'' + \frac{1}{2} (t \a0 +
\ep^{-2/3} \ae)u\Bigr)\ra \\
&\phantom{=}
- 2 \im \psi \la f^{-N} \b0 u, \frac{1}{6}(t \a0 +
\ep^{-2/3} \ae)u\ra \\
&= 
Z_1 + Z_2.
\end{align*}
Taking into account (\ref{eq:7.6}), we obtain
\begin{align*}
|Z_1| 
&\leq \eta \Pi \psi \|\a0 f^{-N + 1/2} u\|_0^2 + \frac{1}{\eta}\Pi
\psi \|f^{-N - 1/2} (u'' + \frac{1}{2}(t \a0 + \ep^{-2/3} \ae)u)
\|_0^2 + |R_1| \\
&\leq (1 + \delta) \eta \Pi \psi \|(t \a0  + \ep^{-2/3} \ae)f^{-N-
  1/2}u\|_0^2
+  2 \Pi \eta \el (C_{a_2}+ Y_2+ C_8 \ep^{4/3} )\|f^{-N - 3/2}u\|_1^2
\\
&\phantom{\leq}
 + \eta (A_3^2+ 2Y_2)\Pi\psi\|f^{-N - 5/2} u\|_0^2 + \frac{1}{\eta}\Pi \psi
 \|f^{-N - 1/2} (u'' + \frac{1}{2}(t \a0 + \ep^{-2/3} \ae)u) \|_0^2 +
 |R_1|
\end{align*}
with $\eta > 0$ which  will be chosen below.
Here and below we note by $R_j, j = 0, 1, 2,...$ lower order terms
which include first order operators. The analysis of these terms  will
be considered in the next subsection.

On the other hand, for $Z_2$, we get
\begin{align*}
|Z_2| &\leq \frac{1}{3} \psi|\im \la f^{-N} \b0  u, f^{-N}(t \a0 +
\ep^{-2/3} \ae) u\ra |
\\
&\leq \frac{1}{3}\psi \|\b0 f^{-N + 1/2} u\|_0 \|(t \a0 +
\ep^{-2/3}\ae) f^{-N - 1/2} u\|_0 + |R_2|
\\
&\leq \frac{1}{3} \Pi \psi \|\a0 f^{-N + 1/2} u\|_0 \|(t \a0 +
\ep^{-2/3} \ae) f^{-N - 1/2} u\|_0 + |R_2|.
\end{align*}
According to (\ref{eq:7.6}), we have
$$
\frac{1}{3} \Pi \psi \|\a0 f^{-N + 1/2} u\|_0 \|(t \a0 + \ep^{-2/3}
\ae) f^{-N - 1/2} u\|_0 
$$
$$ 
\leq \frac{1}{3} \Pi \psi \Bigl[\sqrt{ 1 + \delta} \|(t \a0 + \ep^{-
  2/3} \ae) f^{-N - 1/2} u\|_0 + \sqrt{2(C_{a_2} + Y_3)}
\sqrt{t}\|f^{-N - 3/2} u\|_1   $$
$$
+\sqrt{A_3^2 + 2 Y_2} \|f^{-N - 5/2} u\|_0\Bigr] \|(t \a0 + \ep^{-2/3}
\ae) \fNt u\|_0
$$
$$
\leq \frac{1}{3} \Pi \psi \Bigl( \sqrt{1 + \delta} + 2 \delta_1\Bigr)
\|(t \a0 + \ep^{-2/3} \ae)\fNt u\|_0^2 + \frac{2}{3} \Pi
\delta_1^{-1}(C_{a_2} + Y_3) \el\|f^{-N - 3/2} u\|_1^2
$$
$$ + \delta_1^{-1} 2(A_3^2 + 2 Y_2)/3\Pi \psi \|f^{-N - 5/2} u\|_0^2
$$
with some constant $Y_3$ and $\delta_1 > 0$ such that $\sqrt{1 + \delta} + 2 \delta_1 = 1 + \delta.$
To absorb the leading terms in the sum $\sum_{\nu = 1}^3
\tilde{B}_{\nu}$, the inequalities 
$$
\frac{1}{4 \eta} \Pi \leq \frac{N} {9}, \qquad   (\frac{1}{3} + \eta) (1 +
\delta) \Pi \leq \frac{N}{9}.
$$
must be satisfied to compensate for the terms
$$
\frac{1}{\eta}\Pi \psi \|\fNt (u'' + \frac{1}{2} (t \a0 + \ep^{-2/3}
\ae)u\|_0^2
\quad \text{ and } \quad
(\frac{1}{3} + \eta)(1+ \delta) \Pi \psi \| (t \a0 + \ep^{-2/3} \ae)
\fNt u\|_0^2.
$$
In order to optimize the choice of $\eta$, we take $\eta =
\frac{\sqrt{10} - 1}{6}.$ Then for $N \geq \frac{13}{2} \Pi$ and small
$\delta$ we satisfy the above inequalities. For example, for the first
one we have
$$
\frac{3}{ \sqrt{10} - 1} \leq \frac{13}{9} \Leftrightarrow 27 < 13
\sqrt{10} - 13 \Leftrightarrow 40 < 13 \sqrt{10},
$$
while the second one we get
$$
\frac{\sqrt{10} +1}{6}( 1 + \delta)\leq \frac{13}{18} \Leftrightarrow
3(\sqrt{10} +1) (1 + \delta) \leq 13 \Leftrightarrow 3 \sqrt{10} +
{\mathcal O} (\delta) \leq 10.
$$
Now we fix $\delta$ and so the constants $A_3, Y_2, Y_3$ are fixed. To
absorb $-\frac{2}{3} \delta_1^{-1}(A_3^2+ 2Y_2)\Pi\psi\|f^{-N - 5/2}
u\|_0^2$, 
we exploit the corresponding term in (\ref{eq:final}) and arrange
things so that
$$
\frac{2}{3}\delta_1^{-1}(A_3^2 + 2 Y_2) \Pi \leq \frac{4}{3}
\delta_1^{-1}(A_3^2 + 2 Y_2) \frac{ N}{13} \leq \frac{N(2N+1)}{81},
$$
that is $\frac{27.4}{13} \delta_1^{-1}(A_3^2 + 2 Y_2) \leq 2N + 1.$
Since $N = \frac{13}{2} \Pi + N_0$, we can do this choosing $N_0$
large. In the same way, we arrange the inequality 
$$
\frac{4}{11}\delta_1^{-1} (C_{a_2} + Y_2) \leq \frac{(2N
  +1)C_3}{27}
$$ 
and we absorb the term $- \frac{4}{3}\delta_1^{-1}(C_{a_2}+ Y_3
)\Pi\el\|f^{-N - 3/2}u\|_1^2$ by the corresponding term in
(\ref{eq:final}). 

\subsection{Analysis of $R_j, j= 0,1,2$}
We will prove that we can absorb the terms $R_j$ choosing $\lambda $
large enough.

The term $R_1$ has the form $2 \psi \im \la \alpha_1 f^{-N} u, f^{-N}
(u'' + \fr{1}{3}(t \a0 + \ep^{-2/3} \ae)u)\ra$ with a first order
pseudodifferential operator $\alpha_1$. We have 
$$
R_1 \geq  -\mu \frac{ \el}{t^2} \|f^{-N} (u'' + \fr{1}{3}(t \a0 +
\ep^{-2/3} \ae)u)\|_0^2 - \frac{D_1}{\mu} \el \|f^{-N} u\|_1^2
$$
$$ 
\geq -2 \mu \frac{\el}{t^2}\Bigl( \|f^{-N} (u'' + \fr{1}{3}\ep^{-2/3}
\ae u)\|_0^2  + \fr{1}{9} \|f^{-N}t \a0 u\|_0^2\Bigr) -
\frac{D_1}{\mu} \el \|f^{-N} u\|_1^2
$$
The term involving $\|f^{-N} (u'' + \fr{1}{3} \ep^{-2/3} \ae u)\|_0^2$
is absorbed by the corresponding term in ${\mathcal Q}_1$ choosing
$\mu > 0$ small, the term $- \fr{2}{9} \mu \|f^{-N} t\a0 u\|_0^2$ is
added to the term ${\mathcal Q}_1$ which will be estimated
below. 
Finally,  the last term is absorbed by $\frac{2C_3}{27}\lambda
\el \|f^{-N - 1} u\|_1^2$ in (\ref{eq:final}) taking $\lambda $
large. 
For $R_2$ we use the inequality
$$
R_2 \geq -\delta_3 \frac{\el}{t^2} \|f^{-N} (t \a0 + \ep^{-2/3}\ae)
u\|_0^2 - {\delta_3}^{-1} D_2 \el \|f^{-N} u\|_1^2
$$
$$
\geq - 2\delta_3 \frac{\el}{t^2} \Bigl[t^2 \|f^{-N} \a0 u\|_0^2 +
\ep^{-4/3} \| f^{-N} \ae u\|_0^2\Bigr] - \delta_3^{-1} D_2 \el
\|f^{-N} u\|_1^2.
$$
Taking $\delta_3$ small, we add the first term on the right hand side
to the term ${\mathcal Q}_1$ which we will estimate below. Also the
term with $\ep^{-4/3} \|f^{-N} \ae u\|_0^2$ can be absorbed by
${\mathcal Q}_1$ choosing $\delta_3$ small. The third term on the
right is handled as above. The term $R_0$ is easy to be treated since
we have $f^{-N + 1/2} u$ instead of $f^{-N} u$ and we may repeat the
argument applied for $R_2$. Notice that we choose $\lambda$ large
depending on the norms of first order operators but we keep the
dependence of $N$ on $\Pi$. 

\subsection{Analysis of $t \ep^{2/3} A^{(2)}_{N + 1/2} (u)$}

The term $ 2N t \ep^{2/3} A^{(2)}_{N + 1/2}(u)$ is a sum of
terms. They can be estimated following our previous arguments but we
can take advantage of the factor $t \ep^{2/3}.$ Consider a
typical term:
$$
L_4  = 2N  \ep^{2/3}t \el\re \la f^{-2N - 1} \tilde{a}_2^{\ep} u, u''\ra.
$$
We have modulo lower order terms
$$
|L_4| \leq N \ep^{2/3}\el \Bigl( \|f^{-N - 1/2} u'' \|_0^2 + t^2 C_{a_2}
\|f^{-N- 1/2} u\|_2^2\Bigr).
$$
We absorb this term involving $u''$ by (\ref{eq:final}) taking $\ep$ small and using the term
$N \psi \fr{5}{27} \|\fNt u''\|_0^2$ with small $t$. For the other term we take $\ep$ small to arrange  $\ep C_{a_2} < 1 $ and we apply Proposition 6.1 to reduce the analysis to an estimate of $t^2 C \|\a0 \fNt u\|_0^2$ where we have a factor $t^2$ and the term  can be handled by term in (\ref{eq:final}) involving $t \|\a0 \fNt u\|_0^2.$\\

Next consider the term
$$
L_3 = 2 N \ep^{2/3} t \re \la f^{-2N - 1} \tilde{a}_2^{\ep} u',
u'\ra.
$$
Modulo lower order terms we have
$$
|L_3| \leq N t \ep^{2/3} C_{a_2} \|f^{-N - 1/2} u'\|_1^2
$$
and this can be absorbed by the corresponding terms in (\ref{eq:final}) taking $\ep$ and $t$ small.

\subsection{ Analysis of ${\mathcal Q}_1, {\mathscr R}_3$}

In  ${\mathcal Q}_1$ we have only one negative term. We use the
inequality 
\begin{multline*}
\el\|f^{-N} \a0 u\|_0^2 = \el \la f^{-N + 1/2} \a0 u, f^{- N - 1/2}
\a0 u\ra 
\\
\leq \psi \| \a0 f^{-N + 1/2} u\|_0^2 + \el t \|\a0 f^{-N -
  1/2} u\|_0^2 + R_3.
\end{multline*}
For the first term on the right hand side we apply Lemma 7.1 and we
obtain a leading term $D_3t \|\a0 f^{-N - 1/2} u\|_0^2$ which can be
absorbed taking $N_0$ large. The other terms can be treated as above
exploiting the corresponding terms in (\ref{eq:final}). 
Passing to the term ${\mathscr R}_3$, we get
\begin{align*}
|{\mathscr R}_3 | &\leq \frac{2}{3} \ep^{-2/3} \el \|f^{-N}
(a_2^{\ep})_t u\|_0  \|f^{-N}\ae u \|_0 \\
&\leq \frac{1}{3} \ep^{2/3} \Bigl( \el t\|f^{-N} \pa_t (a_2) u\|_0^2 +
\ep^{-4/3} \frac{\el}{t^2}\|f^{-N} \ae u\|_0^2\Bigr).
\end{align*}
Here the factor $\ep^{2/3}$ comes from the derivative with respect to
$t$ of $a_2(\ep^{2/3} t, \ep x, \xi).$ For small $\ep$ we may absorb
the term involving $\ae u$ using ${\mathcal Q}_1$, while for the other
term we write $\pa_t (a_2) = \gamma_0 a_2$ with a zero order operator
$\gamma_0$ and absorb this term with the corresponding term in $\lambda
{\mathscr E}_N(u)$ taking $\lambda$ large.

\subsection{Analysis of $-2 \ep^{1/3} \frac{4N}{3} t\el\la f^{-2N - 1}
  \be u, u'\ra$ and ${\mathscr R}_j, j = 7,8$}

It is clear that we have
\begin{eqnarray*}
 \ep^{1/3} \frac{4N}{3} t\el|\la f^{-2N - 1}  \be u, u'\ra| \\
\leq 
\ep^{1/3}\fr{4N}{3} t\el \||D_x|^{-1}\be f^{-N-1/2} u\|_0 \||D_x| \fNt u'\|_0 + |R_4|\\
\leq \ep^{1/3} N \el\Bigl[ t^2 D_3\Bigl ( \|\a0 f^{-N-1/2} u\|_0^2 + \|\fNt u\|_0^2\Bigr) 
 + \fr{4}{3}\|\fNt u'\|_1^2 \Bigr] + |R_4|.
\end{eqnarray*}
 We may absorb the term $\el t^2 D_3\|\a0 \fNt u\|_0^2$ on the
right hand side  by the corresponding terms in (\ref{eq:final}) choosing $\ep$ and $t$ small so that $\ep^{1/3} N D_3 t < 1.$
The term with $t^2 \| \fNt u\|_0^2$ is also easily absorbed. Finally, the term 
$\| \fNt u'\|_1^2$ is absorbed by $\frac{10 N}{27}\el \la\a0 \fNt u', \fNt u' \ra$ in (\ref{eq:final}) applying Proposition 6.1 and taking $\ep$ small. The rest $R_4$
involves a second order operator $\gamma_2$ in the place of $b_3$. We apply the same argument and we are going to absorb a term $\el  t^2 D_3 N\|\fNt u \|_1^2$ choosing  $t$ small or by using the term ${\mathcal O}(N^2) \el \|f^{-N - 3/2}u\|_1^2$ in (\ref{eq:final}).\\

The term ${\mathscr R}_8 = - 2 \ep^{1/3} t \el \im \la f^{-2N} (\pa_t \be) u, u'\ra$ can be treated as above. Here we do not have a coefficient $N$ to deal with and, moreover, we have the operator $f^{-2N}$ instead of $\fNt.$ We get
$$|{\mathscr R}_8| \leq \ep^{1/3}  \el\Bigl[ t^2 D_3\Bigl ( \|\a0 f^{-N} u\|_0^2 + \|f^{-N} u\|_0^2\Bigr) 
 + \|f^{-N} u'\|_1^2 \Bigr] + |R_5|.$$
We may absorb all terms taking $\lambda$ sufficiently large, $\ep$ small by using the positive terms in $\lambda E_N(u)$.\\
Passing to the term ${\mathscr R}_7$, consider first
$$\ep^{1/3} \el |\im \la f^{-2N} \be u, u'\ra | \leq \ep^{1/3} D_4 \el \Bigl[\|\a0 f^{-N+ 1/2} u\|_0^2 + \|f^{-N - 1/2} u'\|_1^2 \Bigr] + |R_6|,$$
where $R_6$ includes second order operator coming from the commutator with $b_3$.
For the term $\el \|\a0 \fNt u\|_0^2$ we apply Lemma 7.1 and we reduce the analysis to an estimate of  $\ep^{1/3}D_4 t^2(1 + \delta)\el \|\fNt u \|_0^2$ plus lower order terms. Next we absorb the leading term taking
$\ep^{1/3} D_4$ and $t$ small and using (\ref{eq:final}), where we have positive terms multiplied by $N_0$. The analysis of the other terms
follows the same argument as above. To deal with second summand in ${\mathscr R}_7$, we use the fact that $b_3^{\ep} (t, x, D_x)$ is self-adjoint. Then
$$2\im \la f^{-2N} \be u', u'\ra = -i\Bigl[ \la \be f^{-N} u', f^{-N} u'\ra - \la f^{-N} u', b_3^{\ep} f^{-N} u'\ra $$
$$+ \la \gamma_2 f^{-N} u', f^{-N} u'\ra + \la f^{-N} u', \gamma_3 f^{-N} u'\ra \Bigr]= -i \Bigl[\la \gamma_2 f^{-N} u', f^{-N} u'\ra + \la f^{-N} u', \gamma_3 f^{-N} u'\ra \Bigr]$$
with some second order operators $\gamma_2, \gamma_3$. The analysis of the terms of the right hand side is easy taking the parameter $\lambda$ large in $\lambda E_N(u)$.

\subsection{Analysis of ${\mathscr R}_j, j = 9, 10$}

For the term ${\mathscr R}_9$ we have
$$|{\mathscr R}_9| = \fr{2}{3} \ep^{1/3} t^2 \el |\im \la f^{-N}\gamma_0^{\ep} \aee u, f^{-N}\aee u\ra|\leq C \ep^{1/3} \el t^2 \|f^{-N} \aee u\|_0^2$$
 and the right hand term is easy to be absorbed by the positive terms in $\lambda E_N(u)$ taking $\lambda$ large.\\
Next we obtain
$${\mathscr R}_{10}  = \fr{2}{3} t \ep^{-1/3}\el \im \la f^{-2N} \be u, \ae u\ra = \fr{2}{3} t \ep^{-1/3}\el \Bigl[ \im \la \be f^{-N} u, \ae f^{-N}  u\ra$$ 
$$+ \im \la \gamma_2 f^{-N} u, \ae f^{-N} u \ra + \im \la f^{-N} \be u, \gamma_1 f^{-N} u \ra\Bigr]$$
with operators $\gamma_k$ of order $k = 1,2$ coming from the commutators of $b_3$ and $\ae$ with $f^{-N}.$ Since the symbol $\ae$ has coefficient $\ep^2$, for the last two terms on the right we obtain a factor $\ep^{2/3}$ and terms $C\ep^{2/3} t \|f^{-N} u\|_2^2$ which can be estimated by  $C \ep^{2/3} t\|\a0 f^{-N} u\|_0^2$ and lower order terms. The leading contribution is absorbed by $\lambda E_N(u)$ with large $\lambda$. To treat the term with $\be$ and $\ae$, we exploit the fact that these operators are self-adjoint and we reduce the analysis to an estimation of
$$t e^{-1/3} \el \la [\be, \ae] f^{-N}u ,f^{-N} u\ra.$$
From the commutator $[\be, \ae]$ we obtain a factor $\ep$ and a forth order operator. Thus we are going to estimate $C t \ep^{2/3}\el\|f^{-N} u\|_2^2$ and we proceed as above.

\subsection{Analysis of ${\mathscr R}_j, j = 4, 5, 6$}
We have

$${\mathscr R}_6 = - \fr{1}{3}\el \Bigl( 2 \re \la f^{-2N} u'', (a_2)_t u \ra - \re \la f^{-2N} u', (a_2)_t u'\ra\Bigr).$$ 
The term involving $u'$ is easy to be treated by using $\lambda E_N(u)$. For the term with $u''$ we write
$$2 \el \re \la f^{-N} u'', f^{-N} (a_2)_t u \ra = 2 \el \re \la f^{-N} u'', (a_2)_t f^{-N} u \ra + R_7.$$
Next
$$2\el |\re \la f^{-N} u'', (a_2)_t f^{-N} u \ra | \leq \psi  \|f^{-N} u''\|_0^2 + C  t \el \|f^{-N} u\|_2^2.$$
We may absorb both terms on the right hand side taking $\lambda$ large in $\lambda E_N(u)$ and using Proposition 6.1 to estimate $t\|f^{-N}u\|_2^2$ by  $C_1t\|\aee f^{-N} u\|_0^2$ plus lower order terms. The lower order term $R_7$ is easy to be treated by a similar argument.\\

Passing to the analysis of  ${\mathscr R}_5$, notice that a typical term is $\ep^{-2/3} \psi\re \la \beta_1 f^{-N} u', f^{-N} u'' \ra.$
Here the first order operator $\beta_1$ comes from the operator $\ep^{-2/3} \psi\re \la \beta_1 f^{-N} u', f^{-N} u'' \ra.$ and we get a power of $\ep$ from the commutator $f^{2N}[f^{-2N}, \ae]$ (see (5.10)).
Thus we must estimate
$$\ep^{1/3} \Bigl(C \delta^{-1}\el\|f^{-N} u'\|_1^2 + \delta\fr{\el}{t^2} \|f^{-N} u''\|_0^2\Bigr), \: \delta > 0.$$
For the term involving $\|f^{-N} u'' \|_0^2$ we take $\delta$ small and use the corresponding positive term in ${\mathscr R}_1 + {\mathscr R}_2$, while for the term including $\|f^{-N} u'\|_1^2$ we exploit $\lambda E_N(u)$ with large $\lambda$.\\
  For  $\ep^{-2/3} \psi\re \la \gamma_1 f^{-N} u', f^{-N} u'' \ra$ we repeat the same argument since first order operator $\gamma_1$  comes from $[f^{-2N}, \ae] f^{2N}$. The analysis of 
$\fr{4}{3} \el \re \la f^{-2N} \hat{\alpha}_1^{\ep} u', u''\ra$ is easer since we do not have a factor $\psi$. Finally, the analysis of ${\mathscr R}_4$ follows the same argument as above and  all terms in ${\mathscr R}_4$ can be absorbed by $\lambda E_N(u)$.

Thus we finished the analysis of all terms in $\sum_{j=1}^{10} {\mathscr R}_j + \sum_{\nu = 1}^3 B_{\nu}.$

\subsection{ Analysis of lower order terms} 

The analysis of lower order term in (\ref{eq:final}) is easy since they are generated by lower order terms including only derivatives $D_t^2,\:D^2_{t, x_{j}},\:D_t,\: D_{x_j}$, etc. For this purpose we may use a part of $\lambda E_N(u)$ for example $\fr{\lambda}{3} E_N(u)$ and  take into account the estimate (\ref{eq:mul}), where we have a term 
$$\psi \|f^{-k} ( u'' + \fr{1}{3}(t \aee + \ep^{-2/3} \ae)u)\|_0^2$$
which will appear with a big coefficient $\fr{\lambda}{3}$.
For example, we have
$$ | 2 \im \psi \la f^{-2N} u_{t, x_j},  u'' + \fr{1}{3}(t \aee + \ep^{-2/3} \ae)u\ra | $$
$$\leq \psi \|f^{-N} u'\|_1^2 + \psi \|f^{-N} (u'' + \fr{1}{3}(t \aee + \ep^{-2/3} \ae)u)\|_0^2$$
and we have a control with big parameter ${\mathcal O}(\lambda)$ of both terms in the right hand side of $\fr{\lambda}{3} E_N(u).$ The analysis of other terms is completely similar and we leave the details to the reader. The terms with second derivatives $D^2_{x_i, x_j}$ cannot be treated by this argument and for this purpose we have examined $\sum_{\nu = 1}^3 B_{\nu}$ by a more sophisticated technical tools.
This completes the estimate of the errors terms.\\

As a consequence,(\ref{eq:final}) can be rewritten as a true energy estimate as

\begin{eqnarray}
\label{eq:7.7}
7 \el \| f^{-N} \mathcal{P}u \|^{2} \geq 
\partial_{t} \bigg( \mathscr{E}_{N} (u) +  \frac{N}{27} \psi \| f^{-N-1}  u' \|_0^{2} + \frac{N}{27}  \psi \| f^{-N-2} u \|_0^{2}\notag \\
  + \frac{N}{27} C_3  \el  \| f^{-N-1} u \|^{2}_{1} \bigg)
+ \lambda K_1 E_{N}(u) 
\notag\\
+  K_2 \psi \Bigl[\|f^{-N-1/2}u''\|_0^2 + \| f^{-N-1/2} u'\|_1^2 \
 + t\|f^{-N-1/2} u\|_2^2\Bigr]
\notag\\ 
+ N K_3 \Bigl[\psi \|f^{-N- 1} u'\|_0^2  + N\psi\|f^{-N - 3/2} u'\|_0^2 + \psi\|f^{-N- 2} u\|_0^2  \notag \\
+ \psi N  \|f^{-N-5/2} u\|_0^2
+ t \el\|f^{-N - 3/2} u\|_1^2\Bigr]
\notag \\
+ \lambda N K_4 \psi  \bigg\{ \|f^{-N - 1} u'\|_0^2 + \|f^{-N - 2} u\|_0^2 + t \|f^{-N - 1} u\|_1^2\bigg\}, 
\end{eqnarray}
where  $ K_{j},\: j = 1,...,4, $ are suitable positive constants independent of $
\lambda $ and $ N $.

\subsection{Estimates of the terms on the boundary $s = T$}

\def\eT{e^{-\lambda T}}
Assume that $0 \leq s < T \leq 1$ and $u = D_t u = D_t^2 u = 0$ when
$t = s.$ We take $T$ sufficiently small and we integrate in (\ref{eq:7.7}) from $t$ to $T$ with respect
to $s$. As a result we obtain integrals $\int_t^T(...)ds$ and for $s = T$ the
following boundary terms
\begin{eqnarray} 
\label{eq:7.8}
 {\mathscr E}_N((u(T,.)) + \fr{N}{27}\Bigl( \psi(T) \| f^{-N-1}  u'(T, .) \|_0^{2} +   \psi(T) \| f^{-N-2} u(T, .) \|_0^{2}\notag \\
  +  C_3  e^{-2\lambda T}  \| f^{-N-1} u(T, .) \|^{2}_{1} \Bigr).
\end{eqnarray}
Recall that the argument of the previous section yields
\begin{equation} \label{eq:7.9}
\mathscr{E}_{N}(u(T, .) \geq E_N(u) + T \ep^{2/3} A^{(2)}_{N}(u(T, .))+ 2 e^{-2 \lambda T}\ep^{1/3}T\im \la f^{-N} \be u(T, .), f^{-N} u'(T, .)\ra.
\end{equation}
It is clear that $(T/3 + \la \xi \ra ^{-2/3})^{-N} \leq c_2 
(T/3 + \la \xi \ra^{-2/3})^{-N - 1}$ with $c_2 > 0$ independent on $\xi$ and
$T$. Hence we have a control on the norms $\|(f^{-N} u)(T, .)\|_1,\: \|(f^{-N} u')(T,.)\|_0$, etc. On the other hand, in $E_N(u(T, . ))$ we have a positive term $T\|\a0 (f^{-N}u)(T, .)\|_0^2$. Thus we can repeat the argument of subsection 7.6 to absorb the term involving $b_3$. Since we do not have a big coefficient $N$ in (\ref{eq:7.9}),  it suffices to take $\ep$ and $T$ small. For the term $T\ep^{2/3} A^{(2)}_N(u(T, .))$ we are going to repeat the analysis of subsection 7.4. Taking $\ep$ and $T$ small and exploiting the term $T\|\a0 (f^{-N}u(T, .)\|_0^2$ in $E_N(u(T, .)$, we absorb this term.

Finally, the contribution of the boundary terms is bounded from below
by a positive constant and we may neglect them.

\section{A priori estimate}
%
%
\setcounter{equation}{0}
\setcounter{thm}{0}
\setcounter{prop}{0}  
\setcounter{lem}{0}
\setcounter{cor}{0} 
\setcounter{defn}{0}
\setcounter{rem}{0}
%

 For the function $f^{-1}$  we apply the inequalities
$$f^{-1} = \frac{(1 + |\xi|^2)^{1/3}}{\fr{t}{3}(1 + |\xi|^2)^{1/3} + 1} \leq \frac{(1 + |\xi|^2)^{1/3}}{1 + \fr{t}{3}}, \: t \geq 0,$$
$$f^{-1} \geq \frac{1}{1 + t},\: 0 \leq t \leq T \leq 1.$$
Therefore from (\ref{eq:7.7}) and the analysis in the Section 7 we deduce for $\lambda \geq \lambda_0$ the estimate
\begin{eqnarray} \label{8.1}
\lambda \int_t^T e^{-2\lambda s - 2N \log (1 + s)} \Bigl(\sum_{k=0}^2 s^{1-k}\|\partial_t^k u(s, .)\|_{(2-k)}^2 
+ \sum_{k = 0}^1 \|\partial_t^{k} u(s, .)\|_{(1-k)}^2\Bigr) ds \nonumber \\
\leq C_0 \int_t^T e^{-2\lambda s- 2N \log (1+\fr{s}{3})}\|{\mathcal P} u(s, .)\|^2_{(2N/3)}ds,
\end{eqnarray}
where $\|.\|_{(m)}$ is the $H_{(m)}$ norm in  $\R^n$ for fixed $m.$

Of course, we have a negative power of $s$ only in front of the norm  $\|u''(s, .)\|_0^2$ and we may estimate from below this term without a power of $s$. On the other hand, the norm $\|u(s, .)\|_{(2)}^2$ appears with a coefficient $s$.

\begin{rem} It is not useful to use the estimate $f^{-1}  \leq \frac{3}{t}$ to bound the term $\|f^{-N + 1/2} {\mathcal P}u \|^2$ in $($\ref{eq:7.7}$)$.
If we did then, in $($\ref{8.1}$)$, we would have the integral
$$
\int_t^T e^{-2\lambda s} \frac{1}{s^{2N}}\|{\mathcal P} u(s,
.)\|^2_{(0)}ds
$$
and as $t \to 0$ this would produce no uniform estimates with respect to $t \geq 0.$
\end{rem}

\def\Ls{\Lambda_p}

To estimate the high order derivatives with respect to $x$, consider
the operator $(1 + |D_x|^2)^{\12 p} = \Ls,\: p > 0$ and write 
$$ \Ls P u = D_t^3 ( \Ls u) - t \Bigl(a_2 + [\Ls, a_2] \Ls^{-1}\Bigr)D_t(\Ls u) + \Bigl(b_2  + [\Ls, b_2]\Ls^{-1}\Bigr)(\Ls u) $$
$$ + t \Bigl(a_1 + [\Ls, a_1] \Ls^{-1}) D_t^2(\Ls u) +....$$
Moreover, we observe that the ``perturbations''  $[\Ls, a_2]
\Ls^{-1},\:[\Ls, b_2]\Ls^{-1},\: [\Ls, a_1]\Ls^{-1}$ have order lower
than the terms $a_2, b_2$ and $a_1$, respectively.  Then $v = \Ls u$
satisfies an equation of the type studied above and, moreover,
$(D_t^2 v) (t, x) = (D_t v)(t, x) = v(t, x) = 0.$  Going back to the differential operator $P$, we get the
following 

\begin{thm} 
Assume that $(D_t^2 u) (t, x) = (D_t u)(t, x) = u(t, x) = 0$ and let
$0 \leq t \leq s \leq T$ with a small $T > 0$. Then for every $p \in \R$ there exist $\Delta_p$ and a constant
$C_{p}$  so that for $\lambda \geq \Delta_p, \: N = \fr{13}{2} \Pi +
N_0$ and $u \in C_0^{\infty}(\R^{n + 1})$ we have the estimate 
\begin{eqnarray} 
\label{9.2}
\lambda \int_t^T e^{-2\lambda s- 2N \log(1+s)} \Bigl(\sum_{k = 0}^2 s^{1-k}\|\partial_t^k u(s, .)\|_{(2-k +p)}^2 + \sum_{k= 0}^1 \|\partial_t^k u(s, .)\|_{(1-k + p)}^2 \Bigr)ds \nonumber \\
\leq C_{p} \int_t^T e^{-2\lambda s - 2N \log(1 +\fr{s}{3})}\|P u(s, .)\|^2_{(2N/3 + p)}ds.
\end{eqnarray}
\end{thm}
\def\ell{e^{\lambda t}}
Next we discuss the estimates for functions $u(t, x)$ satisfying the
boundary conditions $D_t^2 u(T, x)$  $ = D_t u(T, x) = u(T, x) = 0$ for $T
> 0$. To do this, we proceed along the same lines as above. We use the function $\varphi(t) = t e^{2\lambda t}$ and we take the scalar
 product of
$f^{2N}(t, D_x){\mathcal P}u$ with the operator $L u = \varphi(t) (D_t^2 - \fr{1}{3} (t \aee u+ \ep^{-2/3} \ae u))$. Thus we 
consider 
$$ 
2 \im \la \varphi(t)  f^{2N}(t, D_{x}) \mathcal{P} u , D_t^2 - \fr{1}{3} (t \aee u + \ep^{-2/3} u) \ra.
$$ 

Notice that we have changed $-\lambda$ to $\lambda$, $f^{-2N}$ to
$f^{2N}$ and we have a $+$ sign in front of the scalar product.
We then handle the terms in the same way as we did in Sections 5-8. For
example, 
$$
2 \Im \frac{1}{i} \la \varphi(t) f^{2N} \partial_t^3 u, \partial_t^2 u \ra
= -2 \Re \la \varphi(t) f^{2N}\partial_t^3 u, \partial_t^2 u \ra 
$$ 
$$
= - \partial_t\Bigl( \varphi(t)  \| f^{2N} \partial_t^2 u \|^2\Bigr) + 2N
\varphi(t) \| f^{N} \partial_t^2 u \|^2 + \varphi'(t) \| f^{N} \partial_t^2 u \|^2.
$$  
Thus we obtain an analog of (\ref{eq:final}) in the form
\def\ell{e^{2\lambda t}}
\begin{eqnarray}
\label{eq:8.2}
7\ell \| f^{N} \mathcal{P}u \|^{2} \geq 
-\partial_{t} \bigg(\mathscr{E}^*_{ N} (u) + \frac{N}{27} \ell \| f^{N-1}  u' \|_0^{2} 
 +\frac{N}{27}  \ell \| f^{N-2} u \|_0^{2} \notag \\
 + \frac{N}{27} C_3 t \ell \| f^{N-1} u \|^{2}_{1}\bigg)\notag \\
+ 2\lambda {\mathscr E}^*_{N}(u) + \ep^{1/3} t\ell \fr{4N}{3} \im \la f^{2N - 1} b_3 u, u'\ra
\notag\\
+  N \varphi \Bigl[\frac{5}{27}\|f^{N-1/2}u''\|_0^2 + \frac{10}{27} \re \la f^{N - 1/2} (t\a0 + \ep^{-2/3}\ae) u' , f^{N-1/2} u'\ra\notag \\
 +\frac{4}{9} \|f^{N - 1/2}(u'' +\fr{1}{2} (t\a0  + \ep^{-2/3} \ae) u)\|_0^2
+ \frac{1}{9} \|f^{N - 1/2} (t \a0 u+ \ep^{-2/3} \ae u)\|_0^2\Bigr]\notag \\
+\frac{N}{9} \Bigl[\frac{\varphi'}{3} \|f^{N - 1} u'\|_0^2 + \frac{2N -5}{9} \varphi\|f^{N - 3/2} u'\|_0^2 + \frac{\varphi'}{3}\|f^{N - 2} u\|_0^2 \notag \\
+ \frac{2N + 1}{9}\varphi \|f^{N-5/2} u\|_0^2 + \frac{2C_3}{3} \lambda \ell \|f^{-N - 1}u\|_1^2 \notag \\
+ \frac{(2N + 1)C_3}{9}\ell\|f^{N - 3/2} u\|_1^2  
+ \frac{ C_3}{3}  \ell \|f^{N - 1/2} u'\|_1^2\Bigr]
\notag \\
+ {\mathcal Q}^*_1 +  \sum_{j=3}^{10} {\mathscr R}^*_{j} + \sum_{\nu = 1}^3 B_{\nu}^*
+ \text{lower order terms},  
\end{eqnarray}
where ${\mathcal Q}^*_1,\:{\mathscr R}^*_j,\: B_{\nu}^*$ are obtained from the corresponding terms
${\mathcal Q}_1,\:{\mathscr R}_j,\:B_{\nu}$ changing $N$ by $- N$ and
\begin{eqnarray}
{\mathscr E}^*_N(u) = \varphi \Bigl[\| f^{N} u''\|^{2} 
                  +(1 - \theta) \re \la f^{2N} (t \aee + \ep^{-2/3} \ae) u', u' \ra \nonumber \\
+ \theta \|f^{N} (t \aee u + \ep^{-2/3} \ae ) u\|_0^2
+ \theta 2\re \la f^{2N} u'', (t \aee + \ep^{-2/3} \ae) u\ra \Bigr]  
\nonumber \\                                                  
+ \epsilon^{1/3} t^3 \ell 2 \im \la f^{2N} b_{3}^{\epsilon} u,
u'\ra .   
\end{eqnarray}

 We repeat the argument of the
previous sections and we integrate with respect to $s$ from $t$ to
$T$, assuming $0 \leq t < T \leq 1.$
Thus we obtain an a priori estimate involving the "weights" $f^{2N - k},\: -1 \leq k \leq 5/2$.

\bigskip

On the other hand,
$$f^{2N} \leq (t +1)^{2N},\: 0 \leq t < T \leq 1,$$
$$ f^{2N}  \geq (1+ \fr{t}{3})^{2N}(1 + |\xi|^2)^{-2N/3}. $$
Consequently, for $\lambda \geq \lambda_0 > 0$, we deduce

\begin{eqnarray} 
\label{eq:finalest}
\lambda \int_t^T e^{2\lambda s + 2N \log(1+\fr{s}{3})} \Bigl(\sum_{k = 0}^2 s^{3-k}\|\partial_t^k u(s, .)\|_{(2- k-2N/3)}^2 + \sum_{k=0}^1 \|\partial_t^k u(s, .)\|_{(1-k-2N/3)}^2\Bigr)ds\nonumber \\
\leq C_0 \int_t^T e^{2\lambda s + 2N\log (1+s)} s^2\|{\mathcal P}u(s, .)\|^2_{(0)}ds.
\end{eqnarray}

Finally, we may make a shift in the Sobolev indices for this estimate
and consider $\|. \|_{(p)}$ norms. Thus we eventually obtain the
following 
\begin{thm} Assume that $(D_t^2 u) (T, x) = (D_t u)(T, x) = u(T, x) = 0$ and let $0 \leq t \leq s \leq T$ with a small $T > 0$. Then for every $p \in \R$ there exist $\Delta_p$ and a constant $C_{p}$  so that for $\lambda \geq \Delta_p, \: N = \fr{13}{2}\Pi + N_0$ and $u \in C_0^{\infty}(\R^{n +1})$ we have the estimate
\begin{eqnarray} 
\lambda \int_t^T e^{2\lambda s + 2N \log(1+\fr{s}{3})}\Bigl(\sum_{k = 0}^2 s^{3-k}\|\partial_t^k u(s, .)\|_{(2-k+ p)}^2 + \sum_{k=0}^1 \|\partial_t^k u(s, .)\|_{(1-k+ p)}^2\Bigr)ds\nonumber \\
\leq C_{p} \int_t^T e^{2\lambda s + 2N\log(1+ s)} s^2\|P u(s, .)\|^2_{(2N/3 + p)}ds.
\end{eqnarray}
\end{thm}
For the local uniqueness result it is more convenient to have estimates
for the operator ${\mathcal P}^*$, the adjoint to ${\mathcal P}$. We
have 
$$
{\mathcal P}^* = D_t^3 u - (t a_2 + \alpha) D_t  + t^2 b_3 +
\bar{b}_2 + i a_2 + t \alpha_2 + {\text lower\:order\:terms},
$$ 
where $\alpha_2$ is a second order operator with respect to $x$. The subprincipal symbol of ${\mathcal P}^*$ for $\rho = (0, x, \xi)$ has the form
$$
\frac{i}{2} a_2(0, x, \xi) + \bar{b}_2(0, x, \xi) = \overline{p_2'}(0,
x, \xi).
$$
Thus the number $\Pi^*$ corresponding to ${\mathcal P}^*$ coincides
with $\Pi$  and Theorems 8.1 and 8.2 hold for the operator ${\mathcal
  P}^*$ changing, if it is necessary, $\Lambda_p$ and $C_p$.

\bigskip

Applying Theorems 8.1 for $P$ and Theorem 8.2 for $P^*$, we can
establish an existence and uniqueness results for the Cauchy problem
in $G = \{(t, x):\: 0 \leq t \leq T, x \in U_{x_0}\}$ with sufficiently small
$T$. To fix the notations, we say that $f \in H^{loc}_{(q, s)}(G)$ if
$\varphi f \in H_{(q, s)}(G)$ for all $\varphi \in C_0^{\infty}
(\R^{n+1})$ and $g \in H_{(q, s)}(\R^{n+1})$ if 
$$
\|g\|_{(q, s)}^2 = (2\pi)^{-(n + 1)}\int  (1 + |\tau|^2)^q (1 +
\xi|^2)^s |\hat{g}(\tau, \xi)|^2 d\tau d\xi < \infty.
$$ 
Since $P$ is strictly hyperbolic for $0 < t \leq T$, we may repeat
with minor modifications the proof of Theorem 23.4.5 in \cite{h3} to
obtain the following 
\begin{thm} Let $P$ be a differential operators with $C^{\infty}$
  coefficients in $G =[0, T] \times U_{x_0}$ satisfying the hypothesis
  $(H_0) - (H_2)$ and let $V_{x_0} \subset U_{x_0}$. 
For  $T$ sufficiently small and for $f \in H^{loc}_{(0, s)}(G)$ having
support in $\bar{G}$ one can find an unique $u \in H^{loc}_{(2, s +2 -
  2N/3)}(G)$ with support in $\bar{G}$ so that $Pu = f$ in $(0, T)
\times V_{x_o}$. 
\end{thm}
We leave the details to the reader.

\bigskip

In conclusion  the conjecture for strongly hyperbolic operators with
triple characteristics is true for operators satisfying
$(H_0)-(H_2)$.


\section{Appendix}
In this appendix we discuss the existence of a factorization
\begin{equation} \label{eq:9.1}
((\tau - \beta(t, x, \xi))^2 - D(t, x, \xi))(\tau- \gamma(t, x ,\xi))
\end{equation}
of the principal symbol $p_3(t, x, \tau, \xi)$ having the form (\ref{eq:2.1}).  We suppose that in (\ref{eq:9.1}) the symbols $\gamma$ and $\beta$ are smooth functions homogeneous of order 1 in $\xi$, while $D$ is smooth and homogeneous of order 2 in $\xi$. It is clear that the root $\gamma$ must be real-valued. We suppose that (\ref{eq:9.1}) holds in a conic neighborhood
of a point $(0, x_0, \xi_0)$ or for fixed $\xi_0 \neq 0$ and $(t, x)$ in a neighborhood of $(0, x_0).$ For the counterexamples we will discuss the non existence of a factorization for fixed $\xi_0$. The problem is to see if there exists a smooth real root $\gamma(t, x, \xi)$ of $p_3 = 0$ in a neighborhood of $(0, x_0, \xi_0).$\\

Consider the symbol 
\begin{equation} \label{eq:9.2}
p_3 = \tau^3 - (t a_2(t, x, \xi) + \alpha(x, \xi)) \tau + t^2 b_3(t, x ,\xi)
\end{equation}
with $a_2 (t, x, \xi) \geq c |\xi|^2,\: \alpha(x, \xi) \geq 0$ and $\alpha(0, \xi_0) = 0, \:4(t a_2 + \alpha)^3 \geq 27 t^4 b_3^2$ for $t \geq 0.$

\begin{prop} Assume that $b_3(0, 0, \xi_0) \neq 0.$ Then if $p_3(t, x, \tau, \xi_0)$ is factorizable for $(t, x)$ in a neighborhood of $(0, 0),$ there exists a neighborhood  $U \subset \R^n$ of $0$ such that $\alpha(x, \xi_0) = 0,\: \forall x \in U.$
\end{prop}

{\bf Proof.} Assume that there exists a real-valued function $\gamma(t, x, \xi_0)$ which is a solution of $p_3(t, x, \gamma, \xi_0) = 0$.  Assume that there exists a sequence $x_m \to 0$ such that $a(x_m, \xi_0) \neq 0.$
For $t = 0$ we obtain the equation $\gamma^3 - \alpha \gamma = 0$ and there are two possibilities.\\

(i) $\gamma(0, x, \xi_0) = 0.$ Then $ \gamma = t \rho(t, x, \xi_0)$ with a continuous $\rho(t, x, \xi_0)$ and we get for $t > 0$ the equality
$$ t \rho^3 - \rho - \frac{\alpha}{t}\rho + b_3 = 0.$$
If $\lim_{(t,x) \to (0,0)}\rho(t, x, \xi_0) \neq 0,$ choosing  $t = |\alpha(x_m, \xi_0)|^2$ and letting $m \to 0$, we obtain a contradiction. If $\lim_{(t,x) \to (0,0)}\rho(t, x, \xi_0) = 0$, we take $t = |\alpha(x_m, \xi_0)|$ and passing to the limit $m \to \infty$, we obtain a contradiction with the fact that $b_3(0, 0, \xi_0) \neq 0.$\\

(ii) Let $\gamma(0, x, \xi_0) = \sqrt{\alpha(x, \xi_0)}$, provided $\sqrt{\alpha}$ smooth. Then we have $\gamma = \sqrt{\alpha(x, \xi_0)} + t \rho(t, x, \xi_0)$. In this case we obtain for $t > 0$ the equality
$$2 \frac{\alpha}{t}\rho + 3 \sqrt{\alpha} \rho^2 + t \rho^3 - \rho - \frac{\sqrt{\alpha}}{t} + b_3 = 0.$$
We choose $t = |\alpha(x_m, \xi_0)|$ and passing to the limit $m \to \infty$ we obtain a contradiction.\\

It is interesting that we have an inverse result. 
\begin{prop}
The  symbol $(\ref{eq:9.2})$ with $\alpha(x, \xi)  \equiv 0$ is factorizable in a neighborhood of $(0, x_0, \xi_0).$
\end{prop}
{\bf Proof.} 
Clearly, the discriminant $\Delta  = 4 t^3 a_2^3 - 27 t^4 b_3^2$ is positive for small $t > 0$. The three roots of the equation $p_3(t, x, \tau, \xi) = 0$ with respect to $\tau$  have the form (see for instance, \cite{SL})
\begin{equation} \label{eq:9.3}
x_k = -\frac{1}{3} \Bigl( u_k C + \frac{3t a_2}{u_k C}\Bigr), \: k = 1, 2, 3.
\end{equation}
where $u_k$ are the three roots of the equation $u^3 = 1$  and $C$ has the form
$$C = \Bigl( \frac{27 t^2 b_3 + \sqrt{-27 \Delta}}{2} \Bigr)^{1/3}.$$
Our goal is to show that we have a real $C^{\infty}$ smooth root $\gamma(t, x, \xi)$ of $p_3 = 0$ defined for $|t| \leq \epsilon$ and
$(x, \xi)$ in a conic neighborhood of $(x_0, \xi_0).$\\

In our case $C$ becomes
$$ C = 2^{-1/3} \Bigl( 27 t^2 b_3 + 27 \sqrt{-\frac{4}{27} t^3 a_2^3 + t^4 b_3^3}\Bigr)^{1/3}$$
$$=  2^{-1/3}3 \Bigl(t^2 b_3 + \sqrt{- t^3 \alpha^3 + t^4 b_3}\Bigr)^{1/3} 
=  2^{-1/3}3 \Bigl( i (\alpha t)^{3/2}(1 - \frac{b_3^2}{\alpha^3} t)^{1/2} + t^2 b_3\Bigr)^{1/3},$$
where $\alpha = \frac{4^{1/3} a_2}{3}.$ Thus
$$C = 2^{-1/3} 3 (\alpha t)^{1/2} \Bigl(i \Bigl( 1 - \frac{b_3^2}{\alpha^3} t\Bigr)^{1/2} + \frac{b_3}{\alpha^{3/2}} t^{1/2}\Bigr)^{1/3}.$$
To obtain a real root $\gamma(t, x, \xi)$ , we take $u_k = 1$ in (\ref{eq:9.3}) and one deduces
$$C + \frac{3 t a_2}{C} = \frac{ C^2 + 3 t a_2}{C}.$$
Now since $i^{2/3} = (-1)^{1/3} = -1$, we get
$$C^2 + 3t a_2 = \Bigl[ - 9 \alpha 4^{-1/3} \Bigl( (1 - \frac{b_3^2}{ \alpha^3} t )^{1/2} - \frac{i b_3}{\alpha^{3/2}} t^{1/2}\Bigr)^{2/3} + 3 a_2\Bigr] t$$
$$= t \Bigl[ -9 \alpha 4^{-1/3} + 3a_2 + \frac{ 4^{-1/3}6 i b_3}{ \alpha^{1/2}} t^{1/2} + {\mathcal O} (t)\Bigr] = t^{3/2} \Bigl[ \frac{4^{-1/3} 6 i b_3}{ \alpha^{1/2}} + {\mathcal O} (t^{1/2})\Bigr] .$$
Dividing by $C$, we get
$$\gamma = -\frac{1}{3}\frac{C^2 + 3 t a_2}{C} =  t \Bigl[ \frac{ 4^{-1/3} 2^{1/3} 2 b_3}{3 \alpha} + {\mathcal O}(t^{1/2})\Bigr] = t \Bigl[ \frac{ b_3}{a_2} + {\mathcal O}(t^{1/2})\Bigr] .$$
Consequently, the real root $\gamma(t, x, \xi)$ is derivable at $t = 0$ and $\pa_t \gamma\vert_{t = 0}  = \frac{b_3}{a_2}.$ Let $\gamma = t \rho.$ Therefore  $t\rho^3 -  a_2 \rho + b_3 = 0.$

Next, consider the function $F(\rho, t, a_2, b_3) = t\rho^3 -  a_2 \rho + b_3.$ Since
$$\frac{\partial F}{\partial \rho}\big\vert_{t = 0} =   - a_2 \neq 0,$$
 by the implicit function theorem we conclude that for small $t$ the function $\rho(t, a_2, b_3)$ is smooth. This implies that the function $\gamma(t, x, \xi)$ is smooth and we have a factorization

$$p_3 = ((\tau - a(t, x, \xi))^2 - b(t, x, \xi))(\tau - \gamma(t, x, \xi))$$
with $a = -\frac{\gamma}{2}$ and $b = t a_2 - 3 a^2.   $


\end{document}